\documentclass{article}
\usepackage{graphicx} 

\usepackage[margin=1in]{geometry}
\AtBeginDocument{%
  \setlength{\abovedisplayskip}{6pt plus 2pt minus 2pt}%
  \setlength{\belowdisplayskip}{6pt plus 2pt minus 2pt}%
  \setlength{\abovedisplayshortskip}{3pt}%
  \setlength{\belowdisplayshortskip}{3pt}%
}

\usepackage{amsmath}
\usepackage{amsthm}
\usepackage{amsfonts}
\usepackage{amssymb}
\usepackage{xcolor}
\usepackage{appendix}
\usepackage{hyperref}
\hypersetup{
    colorlinks=true,
    linkcolor=blue,
    citecolor=blue,
    urlcolor=black
}
\usepackage{caption}
\newcommand\blfootnote[1]{%
  \begingroup
  \renewcommand\thefootnote{}\footnote{#1}%
  \addtocounter{footnote}{-1}%
  \endgroup
}

\numberwithin{equation}{section}

\newtheorem{theorem}{Theorem}[section]
\newtheorem{lemma}[theorem]{Lemma}

\newtheorem{corollary}[theorem]{Corollary}

\theoremstyle{remark}

\newtheorem{remark}{Remark}

\renewcommand{\le}{\leqslant} 
 
\renewcommand{\leq}{\leqslant} 
\renewcommand{\geq}{\geqslant}

\newcommand{\ind}{\mathbf{1}}

\newcommand{\probc}{\stackrel{\mathrm{P}}{\longrightarrow}}

\newcommand{\cA}{\mathcal{A}}
\newcommand{\cF}{\mathcal{F}}

\newcommand{\bN}{\mathbb{N}}
\newcommand{\bR}{\mathbb{R}}

\newcommand{\bZ}{\mathbb{Z}}        

\DeclareMathOperator{\E}{\mathbb{E}}
\DeclareMathOperator{\pr}{\mathbb{P}}

\newcommand{\Poi}{\mathrm{Poi}}
\newcommand{\Exp}{\mathrm{Exp}}

\title{Stochastic Dynamics of Low Earth Orbit Near Full Capacity}
\author{Priyank Behera \\ Department of Statistics, Clemson University \and Aditya S. Gopalan \\  Edwardson School of Industrial Engineering, Purdue University \\ \href{mailto:asgopala@purdue.edu}{asgopala@purdue.edu} \and Harsha Honnappa \\ Edwardson School of Industrial Engineering, Purdue University \\ \href{mailto:honnappa@purdue.edu}{honnappa@purdue.edu}}
\date{}

\begin{document}

\maketitle

\begin{abstract}
The capacity of Low Earth Orbit (LEO) to support sustained space operations is under increasing pressure. 
The number of tracked objects larger than 10 cm has grown sharply over the past two decades, driven by the rapid deployment of commercial megaconstellations, legacy debris from historical fragmentation events, and the emerging prospect of orbital data centers and other non-traditional payloads. The concentration of activity in preferential altitude bands --- particularly the 500–600 km regime --- raises urgent questions about how much operational capacity remains and how quickly it could be lost. 
Collision avoidance maneuvers for Starlink alone have grown exponentially, with the annual maneuver rate scaling super-linearly with constellation size. 
This trajectory suggests a potential tragedy of the commons, in which the aggregate behavior of many operators degrades the orbital environment for all.

Existing models for assessing orbital capacity and debris evolution are predominantly deterministic: they track mean population counts of intact satellites and fragments using ordinary differential equations (ODEs). 
While these models provide valuable baseline forecasts, they cannot capture the inherent randomness of the LEO environment. 
Individual collision events are probabilistic, the number of fragments produced in any given breakup varies substantially, and launch schedules introduce further variability. 

This paper develops a stochastic extension of the two-compartment Lotka--Volterra model introduced by Bradley and Wein~\cite{bradley2009space}, which tracks two interacting species --- intact satellites and debris fragments --- and investigates how randomness affects predictions of orbital capacity, critical density thresholds, and the onset of collisional cascading.
The Bradley and Wein model captures the essential dynamics through a generalized Lotka--Volterra system. 
Intact satellites are injected at a prescribed launch rate, lost through atmospheric drag, and destroyed through collisions with other intacts or with fragments. Fragments are created by those same collisions and removed through drag. 
The ODE system has a unique equilibrium representing the carrying capacity of the orbit --- the maximum number of intact satellites the shell can support without triggering runaway debris growth. 
This equilibrium is governed by a criticality parameter, $\Delta$, defined as the difference between the fragment drag rate and the product of the intact-fragment collision efficiency and the equilibrium intact count; this is also defined as the carrying capacity of the orbit. 
When $\Delta$ is positive, drag exceeds fragment creation and the orbit is stable. When $\Delta$ reaches zero, the system is at critical carrying capacity. 
When $\Delta$ becomes negative, collisional cascading begins.

{
If LEO is operated near its capacity, we expect that the numbers of intacts and fragments differ by many orders of magnitude.
In this paper, we study different deterministic and stochastic limits of a certain density-dependent Markov chain which is in the spirit of the Bradley and Wein~\cite{bradley2009space} model.
Due to the difference in order of magnitude of the intact and fragment counts, limits occur on different time-scales: depending on the type of randomness one wishes to study, one needs to examine different time horizons.
There are several types of limits (depending on the time-scale), and thus several distinct pathways through which a collisional cascade could begin.

In this paper, we specifically study the situation when there are many more fragments than intacts.
Here, intact satellite population fluctuates rapidly on a fast time-scale, while fragments evolve slowly, undergoing sustained high-amplitude excursions. Fragments can drift upward over long periods while the intact count appears roughly stable.
Deterministic models are structurally unable to capture this phenomenon.

The multiple time-scale behavior also yields a third time-scale in which a collisional cascade can begin purely from the fluctuation behavior.
This is in stark contrast to results from fluid limits, wherein the collisional cascades are purely a result of the drift exceeding a threshold; functional central limit theorems around these fluid limits also cannot address the fact that fluctuations can cause collisional cascades.

Our results are as follows: on a fast time-scale for the intacts, we obtain both a fluid limit and a functional central limit theorem.
We also obtain a fluid limit and a functional central limit theorem on an intermediate time-scale for the fragments.
On this time-scale, we obtain a drift threshold above which the fluid limit exhibits collisional cascades.
On a third, slow time-scale, we show that the limiting process is a Feller diffusion, whose collisional cascade behavior is stochastic, and determined purely by the fluctuation behavior, instead of by the drift.
This final result is analogous to \emph{heavy traffic} limits in queueing theory, although the result is a Feller diffusion, rather than the traditional Gaussian process.

The practical implications are significant: if debris runaway can occur substantially sooner and with higher probability than deterministic models predict, current capacity assessments may be systematically overoptimistic. 
The fast-slow structure means warning signs may be subtle — the intact population can appear well-behaved while fragments quietly accumulate risk.
Constellation deployment, debris removal investment, and orbital slot allocation should all account for these stochastic effects.
In addition, the slowest time-scale result indicates that planning for debris runaway cannot be a function only of the mean: the variance also plays a key role in determining pathways for debris runaway.
}

\blfootnote{\emph{Acknowledgments:} AG and HH acknowledge a generous Abhi Deshmukh IE Frontiers Grant, from the Edwardson School of Industrial Engineering.}

\end{abstract}

\section{Introduction}
The \emph{Kessler Syndrome}~\cite{kessler1978collision} describes a positive-feedback scenario in Low Earth Orbit (LEO), in which collisions of intact satellites (``intacts'') with space debris fragments (``fragments'')  incite further such collisions, leading to a run-away behavior in the amount of fragments and a significant decrease in the capacity of LEO for intacts.
Traditional analyses of this phenomenon, such as~\cite{bradley2009space, jang2022stability}, use deterministic source-sink models to characterize a \emph{critical Kessler threshold}, above which the creation rate of fragments exceeds their de-orbit rates, causing a runaway behavior.

In this paper, we study the pathways to collisional runaway behavior using deterministic and stochastic approximations to a certain \emph{density-dependent} Markov chain~\cite{ethierkurtz1986} source-sink model of the LEO environment.
Our model includes intact launches, intact and fragment de-orbits, and fragment creation due to intact-intact and intact-fragment catastrophic collisions.
We show that, in addition to the standard critical Kessler threshold obtained from the deterministic approximation, a collisional cascade can be incited purely from the inherent stochasticity of the collision dynamics.
This is of fundamental importance: strategies to mitigate space debris accumulation thus far have only been informed by the average behavior embodied in deterministic models.
The results in this paper demonstrate the need for novel strategies which also account for stochastic variability.

A key feature of this paper is a multi time-scale analysis, which gives a clean mathematical framework to handle the fact that the various components of typical source-sink models are of differing orders of magnitude.
These terms include both the intact and fragment counts, as well as the different dynamics in the system.
For example, launches of new intacts occur frequently, but intact-intact collisions should be exceedingly rare.
These differing orders of magnitude induce three distinct time-scales of interest:
\begin{itemize}
    \item A (``fast'') \emph{intact} time-scale, corresponding to the intact count.
    \item An (``intermediate'') \emph{fragment} time-scale, corresponding to the fragment count.
    \item A (``slow'') time-scale, which captures the long-term effects of fluctuations in the fragment count.
\end{itemize}

We emphasize that the multiple time-scales arise as a consequence of our stochastic model.
For example, consider again the fact that launches of new intacts occur frequently, but intact-intact collisions should be exceedingly rare.
In a deterministic model, this difference arises in the coefficients of the relevant terms of an ordinary differential equation, and so the impact of intact-intact collisions can be overstated in these types of models: the intact-intact collisions do not make a meaningful contribution to the fragment counts in our large-system asymptotic (see, \textit{e.g.}, Theorem~\ref{prop:slow}).
Nevertheless, we can obtain deterministic models in the style of Bradley and Wein~\cite{bradley2009space} as special cases of our analysis (see Remark~\ref{rem:B-W}).
In addition, our result that a collisional cascade can be elicited purely from the stochasticity in the collision dynamics relies crucially on a precise identification of the slow time-scale.

We now discuss the implications of the time-scale separation.
Intuitively, with many more fragments than intacts, fluctuations in the fragment count occur on longer time horizons than those in the intact count.
From the perspective of the intact count (fast time-scale), the fragment count appears frozen at a constant value, and no critical Kessler threshold can be identified.
In order to identify such a threshold, one instead needs to consider the dynamics for the fragment count on a slower (intermediate) time-scale.
See Figure~\ref{fig:two-clocks} for a visual comparison of the intact and fragment counts on the fast and intermediate time-scales.
\begin{figure}[h!]
    \centering
    \includegraphics[width=0.8\linewidth]{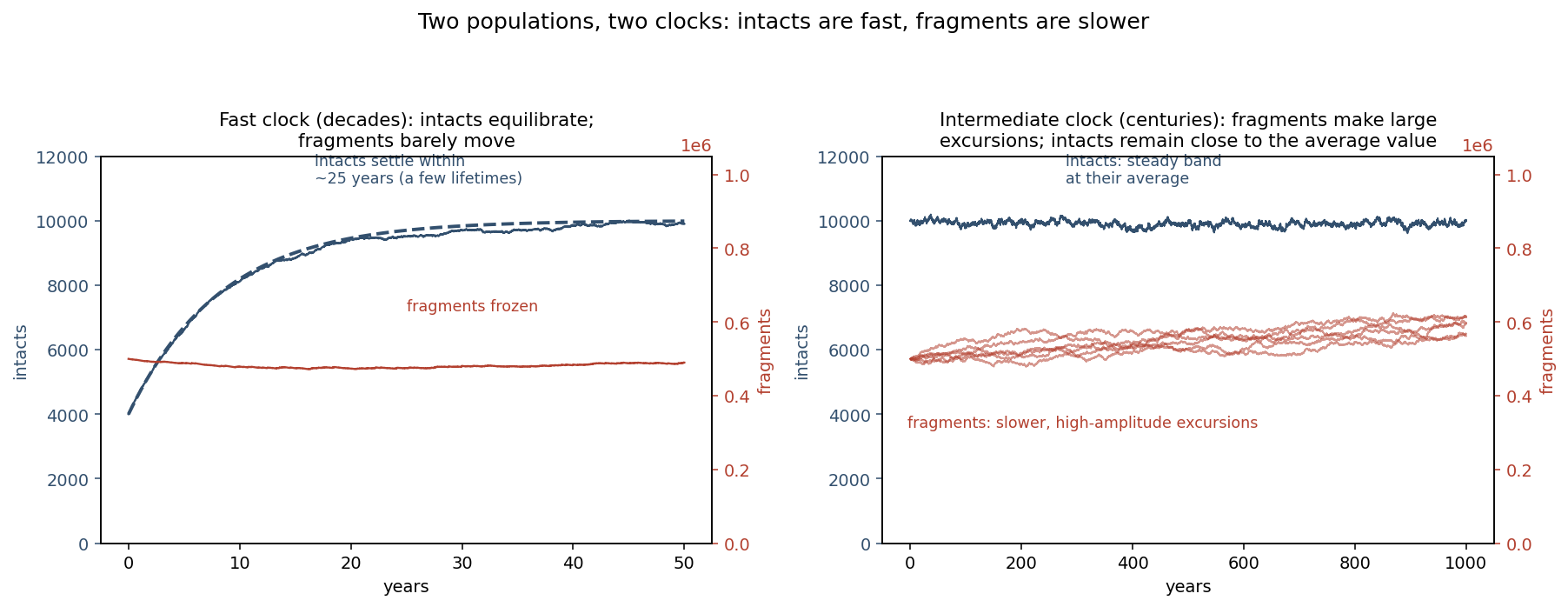}
    \caption{Simulations of intact and fragment population dynamics on the fast and intermediate time-scales, respectively.}
    \label{fig:two-clocks}
\end{figure}

The first of our results is in this direction.
On both the intact (fast) and fragment (intermediate) time-scale, we derive ordinary differential equation (ODE) approximations when the intact and fragment counts are large.
We also derive stochastic differential equation (SDE) approximations to study the effect of stochastic fluctuations around these ODE approximations.
These SDE approximations are Gaussian processes, as is typical for this type of approximation.
The ODE and SDE approximations are different for each of these time-scales.
On the (intermediate) fragment time-scale, we identify a critical Kessler threshold from the ODE approximation.
These results are a significant refinement of the traditional ODE analyses in, \textit{e.g.}, Bradley and Wein~\cite{bradley2009space} due to the incorporation of multiple time-scales.

In the context of understanding the pathways to Kessler Syndrome, these first results fall short: the analysis cannot address the question of whether or not a collisional cascade can begin purely by random chance.

We address this question by identifying a third, slow time-scale.
This slow time-scale captures noise terms that average out in the previous Gaussian process SDE approximation.
The resulting approximation is a \emph{Feller diffusion}.
The key difference between the Gaussian process and the Feller diffusion is that in the Gaussian process, noise is merely additive, and thus it averages out.
In the Feller diffusion, noise \emph{compounds}, and this is what allows us to see the Kessler runaway.

Our most important result is that on the slow time-scale, we show that a Kessler runaway can, in fact, be incited purely from random fluctuations in the collision dynamics.
In this regime, the Kessler runaway is itself a random behavior and is no longer determined by a fixed threshold.
As mentioned above, this has the fundamental operational implication that debris mitigation strategies must account for the variance, and not only the mean.



We note that these analytical results quantify the time-scales on which the capacity region of the LEO environment is reached, but only from the physics-driven dynamics of the intact and fragment counts. These results do not yet have implications for the \emph{carrying capacity} of the LEO environment. However, the time-scale analysis can be applied to models that include behavioral and economic constraints.

\paragraph{\bf Organization of this Paper.}
The rest of this paper is organized as follows.
In Section~\ref{sec:model}, we introduce our compartmentalized model and its underlying assumptions.
In Section~\ref{sec:main-results}, we state our main results and their implications.
The remaining sections prove the main results.


\section{Model}
\label{sec:model}
We consider a particles-in-a-box model of the populations of \emph{intact satellites (intacts)} and \emph{debris fragments (fragments)} in Low Earth Orbit (LEO).
\emph{We count only those fragments which are energetic enough to cause catastrophic collisions with intacts.}
We will make the simplifying assumption that the catastrophic collision of an intact only produces one fragment.

There are 5 components to our dynamics, which can be expressed analogously to a chemical reaction network.
Here, $I$ represents a single intact, and $F$ represents a single fragment.
For example, dynamics~\eqref{eq:rxn1} represents the launch of a new intact, and dynamics~\eqref{eq:rxn3} represents a catastrophic collision between an intact and a fragment, resulting in the creation of a new fragment.
For each of the dynamics below, the constant $\kappa_k$ ($k = 1, \ldots, 5)$ denotes the rate constant for the dynamics.
For example, $\kappa_3$ is the rate at which each intact-fragment pair catastrophically collides.
We use the symbol $\varnothing$ to denote exogenous intact launches in dynamics~\eqref{eq:rxn1}, and intact and fragment de-orbits, respectively, in dynamics~\eqref{eq:rxn4} and~\eqref{eq:rxn5}.
We use the mathematical chemistry terminology that the fragment population is \emph{autocatalytic}: fragments serve as catalysts for the creation of new fragments. Figure~\ref{fig:crn_transition} provides a visual depiction of the dynamics.

\noindent
\begin{minipage}[c]{0.3\textwidth}
\setlength{\abovedisplayskip}{0pt}
\setlength{\belowdisplayskip}{0pt}
\centering
\begin{align}
    \varnothing &\xrightarrow{\kappa_1} I 
    \label{eq:rxn1}
    \\
    2I &\xrightarrow{\kappa_2} 2F
    \label{eq:rxn2}
    \\
    I + F &\xrightarrow{\kappa_3} 2F
    \label{eq:rxn3}
    \\
    I &\xrightarrow{\kappa_4} \varnothing
    \label{eq:rxn4}
    \\
    F &\xrightarrow{\kappa_5} \varnothing
    \label{eq:rxn5}
\end{align}
\end{minipage}
\hfill
\begin{minipage}[c]{0.65\textwidth}
  \centering
   \includegraphics[width=\linewidth]{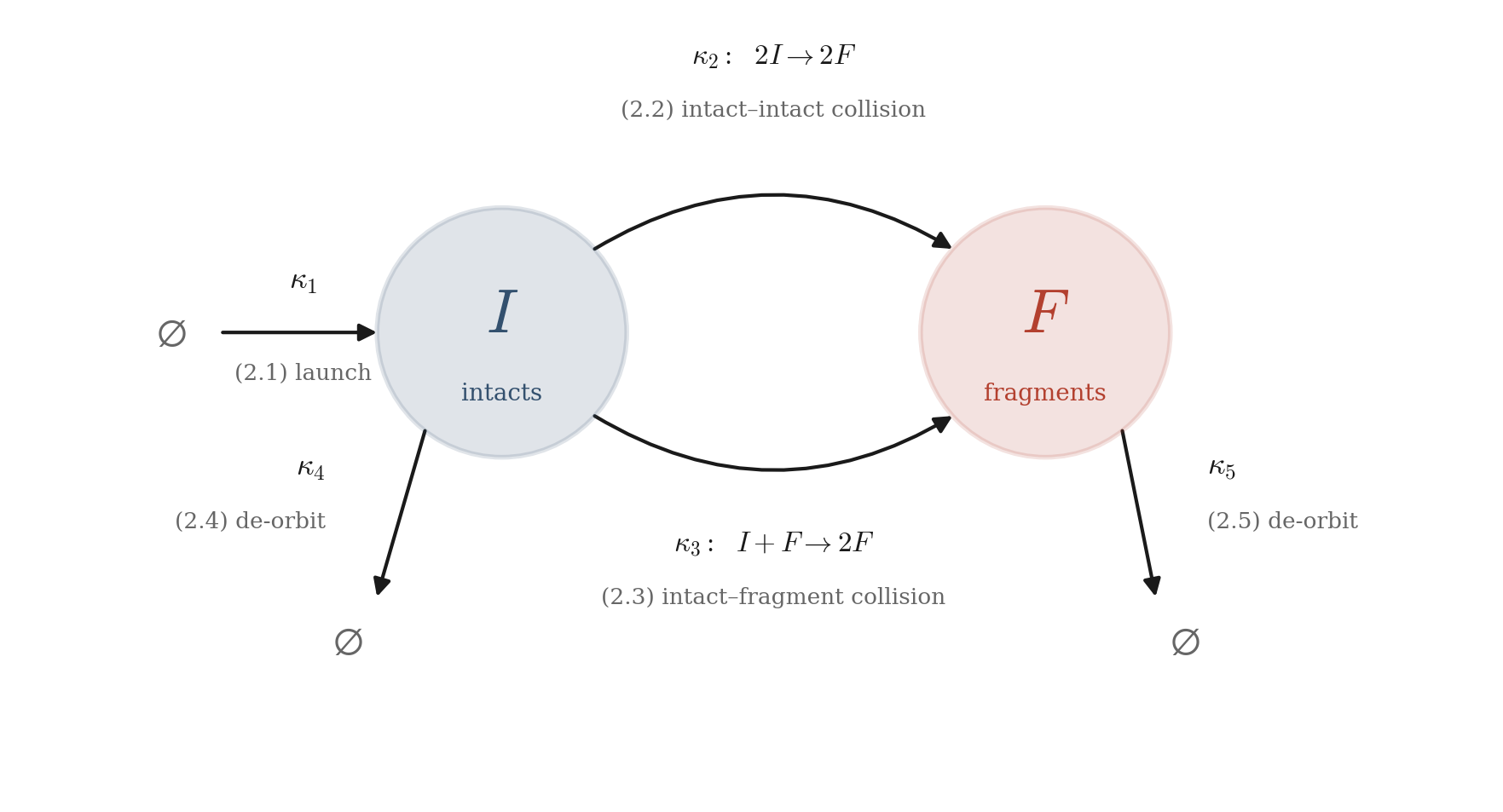}
\end{minipage}%
\par
\captionof{figure}{Transition diagram for the dynamics~\eqref{eq:rxn1}--\eqref{eq:rxn5}. Intacts $I$ enter through exogenous launches~\eqref{eq:rxn1} and are removed either directly by de-orbit~\eqref{eq:rxn4} or by conversion into fragments through intact--intact~\eqref{eq:rxn2} and intact--fragment~\eqref{eq:rxn3} catastrophic collisions; fragments $F$ accumulate from those collisions and are removed only by de-orbit~\eqref{eq:rxn5}. 
}
\label{fig:crn_transition}
\vspace{1em}
The rate constants $\kappa_1, \ldots \kappa_5$ should be interpreted as follows.
$\kappa_1$ is the launch rate and is determined by constellation operators.
$\kappa_4$ and $\kappa_5$ are average de-orbit rates, and depend on the mass and velocity of intacts and fragments; we take $\kappa_4$ and $\kappa_5$ as aggregated over the population of intacts and fragments (resp.). The constants $\kappa_2$ and $\kappa_3$ are to be interpreted as follows.
Assuming that an intact collides with any other intact (respectively, fragment) independently, the intact-intact collision process (respectively, intact-fragment collision process) is approximately a Poisson process when viewed over a sufficiently small time-scale.
The constants $\kappa_2$, $\kappa_3$ include the effects of intact and fragment velocity, as well as the impact of orbital shell height.
One could, for a sufficiently well-occupied orbital shell, treat $\kappa_2$ and $\kappa_3$ as parameters of that shell, but one could also aggregate their values over all of LEO.

For each of the dynamics, we track the \emph{input} and \emph{net change}.
The input tracks the numbers of intacts and fragments needed to produce the dynamics, and the net change tracks the change in the total counts of intacts and fragments should a dynamics actually occur.
The corresponding input vectors are: $\nu_1 := (0, 0)$, $\nu_2 := (2, 0)$, $\nu_3 := (1, 1)$, $\nu_4 := (1, 0)$, $\nu_5 := (0, 1)$;
and the corresponding net-change vectors are: $\zeta_1 := (1, 0)$, $\zeta_2 := (-2, 2)$, $\zeta_3 := (-1, 1)$, $\zeta_4 := (-1, 0)$, $\zeta_5 := (0, -1)$.

Let $X(t) := (X_I(t), X_F(t))$ denote the process of intact and fragment counts: that is, the numbers of intacts and fragments, respectively.
The orders of magnitudes of $X_I$ and $X_F$ are significantly different; therefore, time-scale separation is a relevant phenomenon for studying Kessler Syndrome.
Our goal is to use limiting approximations in the forms of ordinary and stochastic differential equations in order to \emph{quantitatively} explain how and when Kessler Syndrome occurs. 
To do that we need to bring both $X_I$ and $X_F$ to $O(1)$.
To do this, we introduce the system parameter $N$ and the notation $X^N(t) := (X^N_I(t), X^N_F(t))$.
The indexing parameter $N$ captures ``system size'' via \[N^{-\alpha_I}X^N_I(t) = O(1) \quad \text{and} \quad N^{-\alpha_F}X^N_F(t) = O(1),\] for some constants $\alpha_I, \alpha_F$ which do not depend on $N$.
We refer to $\alpha_I$ and $\alpha_F$ as \emph{abundances}.
If $\alpha_I = \alpha_F = 1$, then the abundances can be interpreted as concentrations of intacts or fragments per unit volume; the generalization from concentration to abundance is due to the several orders of magnitude difference between $X_I$ and $X_F$.
Denote $\alpha := (\alpha_I, \alpha_F)$.
When $\alpha_I = \alpha_F = 1$, our system dynamics converge to standard particles-in-a-box ordinary differential equations similar to those in Bradley and Wein~\cite{bradley2009space}.

The rate constants themselves may also vary over several orders of magnitude: for example, one expects that launches (dynamics~\eqref{eq:rxn1}) occur at rate $O(1)$, while the rate of intact-intact catastrophic collisions (dynamics~\eqref{eq:rxn2}) should be exceedingly rare per intact-intact pair.
To handle these order of magnitude differences in the rate constants, let $\beta_1, \ldots, \beta_5$ be such that $N^{\beta_i}\kappa_i = O(1)$, for $i = 1, \ldots,  5$.
Set $\rho_k := \beta_k + \alpha \cdot \nu_k$, for $k = 1, \ldots, 5$: here, $\alpha \cdot \nu$ is the inner product between vectors $\alpha$ and $\nu$.

The system dynamics are as follows:
Let $Y_1,\ldots,Y_5$ be independent unit-rate Poisson processes. Then our system dynamics are given by:
\begin{align*}
X_I^N(t)
&=X_I^N(0)
+Y_1\!\left(\kappa_1 t\right) 
-2Y_2\!\left(\int_0^t \kappa_2 X^N_I(s)\bigl(X^N_I(s)-1\bigr)\,ds\right) 
-Y_3\!\left(\int_0^t \kappa_3 X^N_I(s)X^N_F(s)\,ds\right)
-Y_4\!\left(\int_0^t \kappa_4 X^N_I(s)\,ds\right),
\\[1em]
X^N_F(t)
&=X^N_F(0)
+2Y_2\!\left(\int_0^t \kappa_2 X^N_I(s)\bigl(X^N_I(s)-1\bigr)\,ds\right) 
+Y_3\!\left(\int_0^t \kappa_3 X^N_I(s)X^N_F(s)\,ds\right)
-Y_5\!\left(\int_0^t \kappa_5 X^N_F(s)\,ds\right).
\end{align*}

Define the process $Z^N(t) := (Z^N_I(t), Z^N_F(t)) = \left(N^{-\alpha_I}X^N_I(t),N^{-\alpha_F}X^N_F(t)\right)$ and the re-scaled system dynamics:
\begin{align*}
Z_I^N(t)
&=Z_I^N(0)
+N^{-\alpha_I}
Y_1\!\left(
N^{\beta_1}\kappa_1 t
\right) 
-2N^{-\alpha_I}
Y_2\!\left(
N^{\beta_2+2\alpha_I}\kappa_2
\int_0^t Z_I^N(s)
\left(Z_I^N(s)-N^{-\alpha_I}\right)\,ds
\right) \\
&\quad
-N^{-\alpha_I}
Y_3\!\left(
N^{\beta_3+\alpha_I+\alpha_F}\kappa_3
\int_0^t Z_I^N(s)Z_F^N(s)\,ds
\right)
-N^{-\alpha_I}
Y_4\!\left(
N^{\beta_4+\alpha_I}\kappa_4
\int_0^t Z_I^N(s)\,ds
\right),
\\[1em]
Z_F^N(t)
&=Z_F^N(0)
+2N^{-\alpha_F}
Y_2\!\left(
N^{\beta_2+2\alpha_I}\kappa_2
\int_0^t Z_I^N(s)
\left(Z_I^N(s)-N^{-\alpha_I}\right)\,ds
\right) 
+N^{-\alpha_F}
Y_3\!\left(
N^{\beta_3+\alpha_I+\alpha_F}\kappa_3
\int_0^t Z_I^N(s)Z_F^N(s)\,ds
\right) \\
&\quad
-N^{-\alpha_F}
Y_5\!\left(
N^{\beta_5+\alpha_F}\kappa_5
\int_0^t Z_F^N(s)\,ds
\right).
\end{align*}

In Sections~\ref{sec:fluid-limit} and~\ref{sec:diffusion-limit}, we will identify two time-scales $\gamma_I$ and $\gamma_F$, corresponding to the intacts and fragments (resp.), as well as the relevant ODE and SDE approximations of $(Z_I^N, Z_F^N)$ on those time-scales.
In Section~\ref{sec:critical}, we will consider a different SDE approximation as the system approaches criticality, on a slower $\gamma_F + \alpha_F$ time-scale.


\subsection{Determining Time-Scales $\gamma_I$ and $\gamma_F$}
\label{ssec:time-scales}
First, we will \emph{balance} each of the compartments $I$ and $F$: we require that (up to order of magnitude), the rates of creation of new intacts (respectively, fragments) matches the rates of removal. 
Without this balance, we do not have $Z^N_I(t) = O(1)$ and $Z^N_F(t) = O(1)$, both of which are required to derive fluid and diffusion approximations.

This balance also allows us to determine natural time-scales $\gamma_I$ and $\gamma_F$ for intacts and fragments, respectively, where limiting dynamics for the system are well-defined.
Specifically, we find limits for the processes $Z^{N, \gamma_I}$ and $Z^{N, \gamma_F}$.
Here, \emph{process $Z^{N, \gamma}(t)$ on the $\gamma$ time-scale} is given by:
\[Z^{N, \gamma}(t) := Z^N(N^{\gamma}t) = (Z^N_I(N^\gamma t), Z^N_F(N^\gamma t)).\]
The fact that there are \emph{different} values for $\gamma_I$ and $\gamma_F$ is precisely what is meant by the separation of time-scales.




We now determine the time-scales $\gamma_I$ and $\gamma_F$.
Notice that intacts are only created by dynamics~\eqref{eq:rxn1}, and they are removed by dynamics~\eqref{eq:rxn2},~\eqref{eq:rxn3},~\eqref{eq:rxn4}.
For balance, we therefore require:
\[\rho_1 = \max(\rho_2, \rho_3, \rho_4).\]
We get a natural time-scale $\gamma_I$ for the intact population via:
\[0 \leq \gamma_I := \alpha_I - \rho_1 = \alpha_I - \max( \rho_2, \rho_3, \rho_4).\]
Similarly, fragments are created by dynamics~\eqref{eq:rxn2},~\eqref{eq:rxn3} and are removed only by dynamics~\eqref{eq:rxn5}.
We require:
\[\rho_5 = \max(\rho_2, \rho_3).\]
We get a natural time-scale for the fragment population via:
\[0 \leq \gamma_F := \alpha_F - \rho_5 =  \alpha_F - \max(\rho_2, \rho_3).\]
For technical reasons, we will also require: 
\begin{equation}
    \max(\gamma_I, \gamma_F) \leq \max(\alpha_I, \alpha_F) - \max_{k \in \{1, \ldots, 5\}}\rho_k.
    \label{eq:extra-balance-constraint}
\end{equation}
Briefly, the reason for the constraint~\eqref{eq:extra-balance-constraint} is as follows.
Even if the balance $\rho_1 = \max(\rho_2, \rho_3, \rho_4)$ and $\rho_5 = \max(\rho_2, \rho_3)$ are satisfied, it is still possible that the total rate (up to order of magnitude) of creation of both intacts and fragments does not match the total rate of their de-orbits.
See~\cite{kang2013separation} for an example.
We expect that requirement~\eqref{eq:extra-balance-constraint} should be automatically satisfied for realistic values of the system parameters for LEO operations whenever $\rho_1 = \max(\rho_2, \rho_3, \rho_4)$ and $\rho_5 = \max(\rho_2, \rho_3)$ are satisfied.

\subsection{Assumptions on Parameters}
We make the following modeling assumptions.
First, we assume that the balance for intacts is due to de-orbiting, and \emph{not} due to catastrophic collision.
That is, dynamics~\eqref{eq:rxn1} is balanced by dynamics~\eqref{eq:rxn4} ($\rho_1 = \rho_4$).
Next, we assume that the removal of fragments in dynamics~\eqref{eq:rxn5} is balanced by their creation due to catastrophic collisions of intacts with fragments in dynamics~\eqref{eq:rxn3} ($\rho_5 = \rho_3$; that is, intact-intact catastrophic collisions are very rare compared to intact-fragment catastrophic collisions).

From these, we get the following:
\begin{itemize}
    \item $\beta_1 = \beta_4 + \alpha_I$ (since $\rho_1 = \rho_4)$,
    \item $\beta_1 > \beta_2 + 2\alpha_I$ (since $\rho_1 > \rho_2)$,
    \item $\beta_1 > \beta_3 + \alpha_I + \alpha_F$ (since $\rho_1 > \rho_3$),
    \item $\beta_3 + \alpha_F > \beta_2 + \alpha_I$ (since $\rho_3 > \rho_2$),
    \item $\beta_5 = \beta_3 + \alpha_I$ (since $\rho_5 = \rho_3$).
\end{itemize}

We make two final assumptions. 
First, we assume that the abundance of fragments exceeds the abundance of intacts.
This assumption is so that we can carefully understand dynamics that may lead to Kessler-Syndrome-type runaway behavior.
Second, the typical life time of a fragment exceeds the lifetime of an intact.
That is,
\begin{itemize}
    \item $\alpha_F > \alpha_I$,
    \item $\beta_4 > \beta_5$.
\end{itemize}



\section{Informal Statement of Results}
\label{sec:main-results}

We give informal statements of our main results for the sake of readability to a broader audience.
Formal theorem statements and their proofs are deferred to the remaining sections in this paper.
The full results for each subsection here are in the section with the corresponding title.

\subsection{Intrinsic Versus Extrinsic Time-Scale Separation}

Our analysis is intimately related to time-scale separation in fast-slow dynamics.
\emph{Khasminskii averaging}~\cite{Khasminskij1968OnTP} is perhaps the most widely used technique for time-scale separation between a slow and a fast stochastic process.
In Khasminskii averaging, an external parameter $\varepsilon$ (in the context of our paper, $\varepsilon$ could represent, \textit{e.g.} the volume of an orbital shell; see \cite{choi2026} where such a scaling is presented for a source-sink LEO capacity model) is introduced, and limits are taken as $\varepsilon$ is sent to zero.
The parameter $\varepsilon$ captures the speed difference between the time-scales of two processes, but notably the relative speed is not a fixed quantity.
We will refer to this as \emph{extrinsic} time-scale separation.

We are instead concerned with \emph{intrinsic time-scale separation} in the spirit of Kang and Kurtz~\cite{kang2013separation} and Kang, Kurtz, and Popovic~\cite{kang2014central}.
Here, some mild additional assumptions are required, but we are able to explicitly identify the different time-scales on which different processes and their limits exist.
Operationally, this is a more significant result, as it informs debris mitigation and intact launch strategies.
Further, our result on the Feller diffusion relies crucially on the ability to identify the relative time-scales for the intact and fragment populations, which cannot be done explicitly via extrinsic methods.

\subsection{ODE Approximations on the Fast and Intermediate Time-Scales}
Recall the balance conditions of Section~\ref{ssec:time-scales}, and also recall that we use the index $N$ as a measure of ``system size.''
We start with the fast time-scale.
By our assumptions that $\rho_1 > \rho_2$ and $\rho_1 > \rho_3$, it follows that dynamics~\eqref{eq:rxn2} and~\eqref{eq:rxn3} do not contribute to an ODE approximation on the intact time-scale.
Since dynamics~\eqref{eq:rxn5} is balanced by~\eqref{eq:rxn3}, we expect that the fragment counts should remain constant in the ODE approximation.
Indeed, as long as the initial conditions $Z^N(0)$ converge to $(z_{I, 0}, z_{F, 0})$ with $z_{I, 0} > 0$ and $z_{F, 0} > 0$, we have:
\begin{align*}
    Z_I^{N, \gamma_I}(t) &\to \bar{z}_I(t),
    \\
    \dot{\bar{z}}_I(t) &= \kappa_1 - \kappa_4\bar{z}_I(t), \qquad \bar{z}_I(0) = z_{I, 0};
    \\
    Z_F^{N, \gamma_I}(t) &\to z_{F, 0}.
\end{align*}

Similar reasoning allows us to get the ODE approximation for the intermediate time-scale.
Indeed, the fragments are created by dynamics~\eqref{eq:rxn3}, and that is balanced by dynamics~\eqref{eq:rxn5}.
Each fragment collides with intacts at rate proportional to the numbers of intacts, which are governed by dynamics~\eqref{eq:rxn1} and~\eqref{eq:rxn4} as discussed previously.
Denoting $z_I^\star := \frac{\kappa_1}{\kappa_4}$, we have:
\begin{align*}
    Z_I^{N, \gamma_F}(t) &\to z_I^\star;
    \\
    Z_F^{N, \gamma_I}(t) &\to \bar{z}_F(t),
    \\
    \dot{\bar{z}}_F(t) &= (\kappa_3 z_I^\star - \kappa_5)\bar{z}_F(t), \qquad \bar{z}_F(0) = z_{F, 0}.
\end{align*}

See Figure~\ref{fig:ode-approx} for the convergence to the ODE approximations in Theorems~\ref{prop:fast} and~\ref{prop:slow}.

\begin{figure}[h!]
    \centering
    \includegraphics[width=0.8\linewidth]{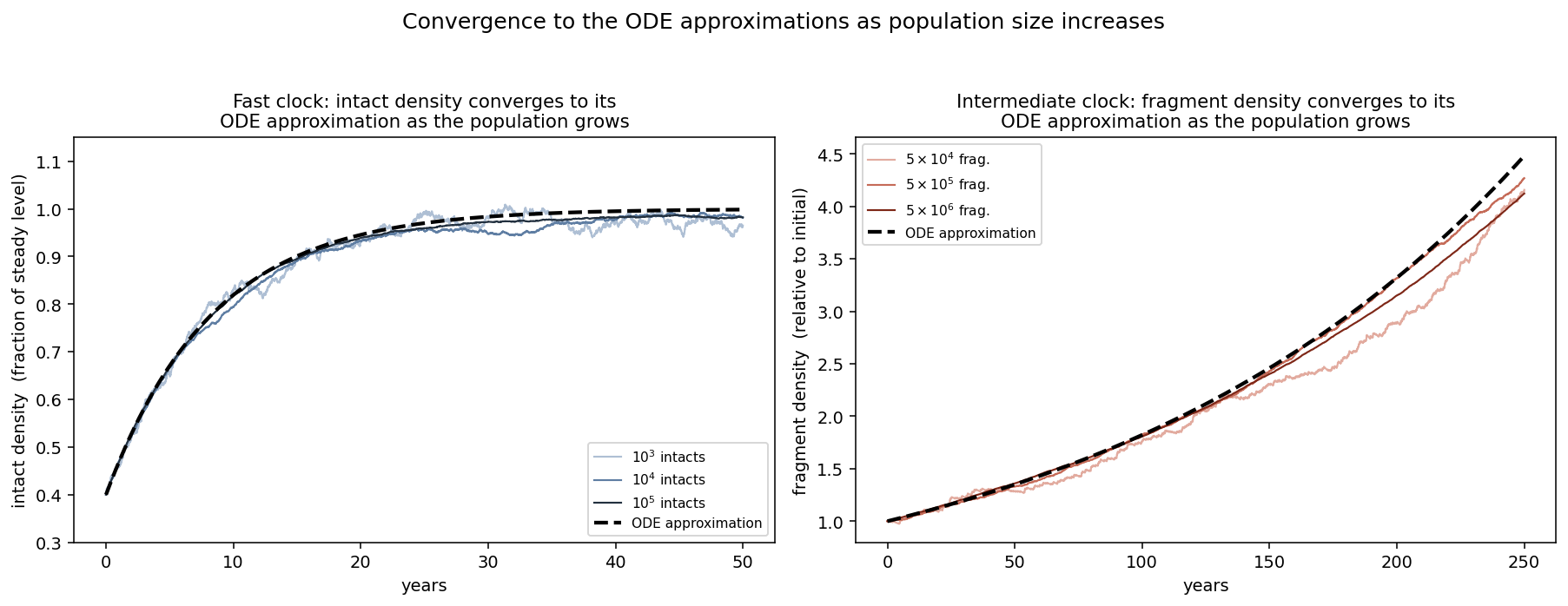}
    \caption{Convergence to the ODE approximation for the intacts on the fast time-scale, and the fragments on the intermediate time-scale.}
    \label{fig:ode-approx}
\end{figure}

We can easily obtain the critical Kessler threshold from this intermediate time-scale limit.
Indeed, observe that the sign of $\kappa_3 z_I^\star - \kappa_5$ exactly determines the behavior of $\bar{z}_F$: the limiting or critical capacity is $\frac{\kappa_5}{\kappa_3}$, and runaway collisions occur when $z_I^\star$ exceeds this threshold.
See Figure~\ref{fig:kessler-threshold} for a depiction of the system dynamics below, above, and at the critical Kessler threshold.

\begin{figure}[h!]
    \centering
    \includegraphics[width=0.8\linewidth]{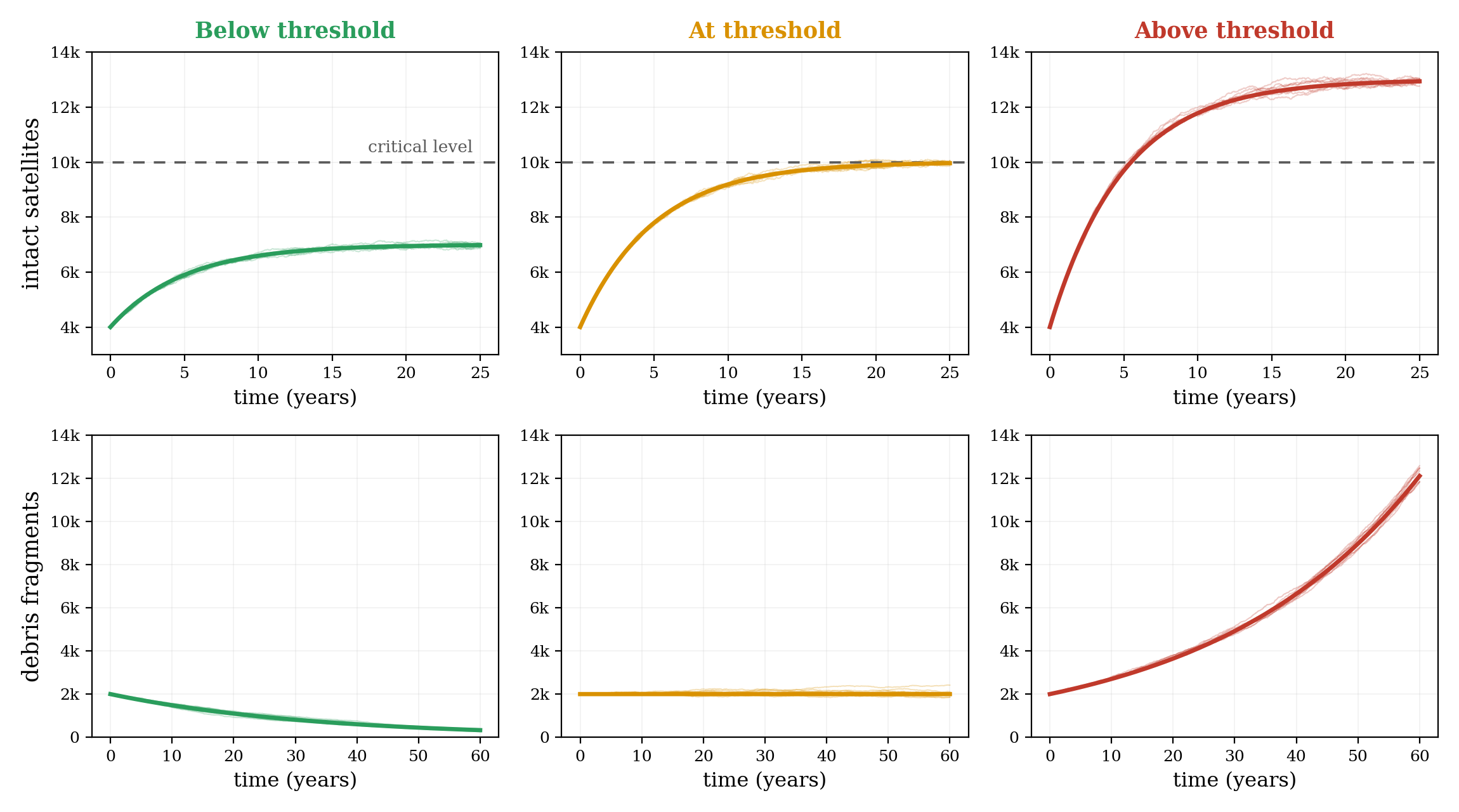}
    \caption{Simulations of the intact and fragments below, at, and above the critical Kessler threshold, respectively.}
    \label{fig:kessler-threshold}
\end{figure}

\subsection{SDE Approximations on the Fast and Intermediate Time-Scales}
Here, we give intuition behind the SDE approximations on the fast and intermediate time-scales, so that the need for the slow time-scale is clear.

The ODE approximations of the previous subsection track only the average behavior of the two populations.
To quantify the random fluctuations \emph{around} these averages, we center each population at its ODE approximation and magnify the difference by the square root of its abundance, $N^{\alpha_I/2}$ for intacts and $N^{\alpha_F/2}$ for fragments.
On each time-scale, the rescaled fluctuations converge to a \emph{Gaussian} process, driven by a standard Brownian motion $B$.
As above, we assume that \[N^{\alpha_I/2}\big(Z_I^N(0) - \bar z_I^N(0)\big) \Rightarrow v_0, \qquad N^{\alpha_F/2}\big(Z_F^N(0) - \bar z_F^N(0)\big) \Rightarrow u_0,\]
where $v_0$ and $u_0$ are random variables, and $\Rightarrow$ denotes weak convergence.

On the fast (intact) time-scale, the intact fluctuations obey a mean-reverting (Ornstein--Uhlenbeck) SDE:
\begin{align*}
    N^{\alpha_I/2}\big(Z_I^{N, \gamma_I}(t) - \bar{z}_I(t)\big) &\to V(t),
    \\
    dV(t) &= -\kappa_4\, V(t)\,dt + \sqrt{\kappa_1 + \kappa_4\, \bar{z}_I(t)}\;dB(t).
\end{align*}
The restoring drift $-\kappa_4 V$ pulls the fluctuations back toward zero: launches and de-orbits continually return the intact population to its ODE trajectory, so the intacts settle into a steady Gaussian band around $\bar{z}_I$.

On the intermediate (fragment) time-scale, the fragment fluctuations obey a linear SDE whose drift is governed by the \emph{same} growth rate $\kappa_3 z_I^\star - \kappa_5$ that governs the fragment ODE:
\begin{align*}
    N^{\alpha_F/2}\big(Z_F^{N, \gamma_F}(t) - \bar{z}_F(t)\big) &\to U(t),
    \\
    dU(t) &= (\kappa_3 z_I^\star - \kappa_5)\, U(t)\,dt + \sqrt{(\kappa_3 z_I^\star + \kappa_5)\, \bar{z}_F(t)}\;dB(t).
\end{align*}
These Gaussian approximations are, however, insufficient to determine if Kessler Syndrome can be caused purely by randomness in the collision dynamics.
The fragment fluctuation $U$ is only a mean-zero correction, of relative size $N^{-\alpha_F/2}$, around the ODE term $\bar{z}_F$.
Away from the critical Kessler threshold --- that is, when $\kappa_3 z_I^\star - \kappa_5 \neq 0$ --- the sign of this rate already decides the fate of the fragment population (decay or runaway), and the Gaussian correction stays uniformly subordinate to $\bar{z}_F$: it merely places a thin random band around the ODE trajectory.

The single regime in which the fluctuations are \emph{not} negligible is exactly at the critical Kessler threshold $z_I^\star = \frac{\kappa_5}{\kappa_3}$.
There the drift $\kappa_3 z_I^\star - \kappa_5$ vanishes, the term $\bar{z}_F$ is constant, and the accumulating noise is no longer held in check by any restoring force.
The fluctuations then grow to the order of $\bar{z}_F$ only after a time horizon of order $N^{\alpha_F}$.
This is precisely why the Gaussian SDE approximations cannot, on their own, determine whether a collisional cascade can be ignited purely by randomness in the collision dynamics: at criticality the decisive dynamics live on the longer time-scale $\gamma_F + \alpha_F$, where, as we show next, the correct limit is a \emph{Feller diffusion}.

\subsection{Slow Time-Scale Limit at the Critical Kessler Threshold}

We consider a \emph{near-critical family}, in which the rate constants $\kappa_k^{(N)}$ vary with $N$ so that the intact steady state approaches the critical abundance $z_I^{\mathrm{crit}} := \frac{\kappa_5}{\kappa_3}$ at exactly the right speed.
The single parameter that survives in the limit is the rescaled distance to the threshold,
\[
    \theta := \lim_{N} N^{\alpha_F}\, \Delta_N, \qquad \Delta_N := \kappa_3^{(N)}\, z_I^{\star, N} - \kappa_5^{(N)} .
\]
On this slow time-scale, the time-changed fragment density converges to a \emph{Feller diffusion} absorbed at $0$:
\begin{align}
    Z_F^{N, \gamma_F}\big(N^{\alpha_F}\tau\big) &\to z_F(\tau),
    \nonumber
    \\
    dz_F(\tau) &= \theta\, z_F(\tau)\,d\tau + \sqrt{2\kappa_5\, z_F(\tau)}\;dB(\tau), \qquad z_F(0) = z_{F, 0} .
    \label{eq:feller-diffusion}
\end{align}

The Feller diffusion differs from the Gaussian approximations in one essential way.
In the Gaussian process the noise is \emph{additive}: it sits on top of the deterministic trajectory and averages out due to symmetry.
In the Feller diffusion the noise is \emph{multiplicative} --- its size $\sqrt{2\kappa_5 z_F}$ grows with the fragment count --- so fluctuations \emph{compound} rather than cancel.
An upward excursion enlarges the noise, which enables a further excursion, and so on.
This compounding is precisely what allows a collisional cascade to start purely from the randomness in the collisional dynamics, and requires the multiplicative noise of the Feller diffusion.

The parameter $\theta$ now plays the role that the sign of $\kappa_3 z_I^\star - \kappa_5$ played in the ODE approximation, except that the deterministic sign test is replaced by a \emph{probabilistic} one.
Writing $z_{F, 0}$ for the initial fragment density:
\begin{itemize}
    \item \emph{Below the threshold} ($\theta < 0$): the fragment population still goes extinct, as the ODE predicts --- but before clearing, it can make large random excursions, growing by a factor of order $\frac{\kappa_5}{|\theta|\, z_{F, 0}}$ with probability bounded away from zero. This excursion scale diverges as the threshold is approached.
    \item \emph{At the threshold} ($\theta = 0$): the deterministic drift in~\eqref{eq:feller-diffusion} is exactly zero, yet the population reaches any prescribed level $L$ with probability $\frac{z_{F, 0}}{L}$, purely from the noise in the collision dynamics. Extinction still occurs almost surely, but the time to extinction is heavy-tailed.
    \item \emph{Above the threshold} ($\theta > 0$): the population survives with positive probability $1 - e^{-\theta z_{F, 0}/\kappa_5}$, and on the event of survival it grows exponentially --- a stochastic Kessler runaway.
\end{itemize}

Thus, near the critical threshold, whether the debris field runs away is itself random, governed by the randomness in the collisional dynamics rather than by the ODE drift.
Operationally, this means that planning for a Kessler cascade cannot rely on the mean behavior alone: the variance of the collision dynamics is itself a determinant of whether runaway occurs.

It is important to note that the convergence to the Feller diffusion is the same type of convergence used for the ODEs on the fast and intermediate time-scales, other than the time-scaling.
Even though the Feller diffusion is also an SDE, this is a fundamentally different result than the Gaussian SDEs on the fast and intermediate time-scales, which are obtained by zooming into the fluctuations around the ODE.
This is, at an intuitive level, why the Feller diffusion can capture the stochastic Kessler Syndrome, but the Gaussian SDE cannot.

See Figure~\ref{fig:gaussian-feller} for a comparison of the Feller diffusion to the Gaussian process approximation and Figure~\ref{fig:runaway-probability} for a depiction of runaway probabilities.

\begin{figure}[t]
    \centering
    \includegraphics[width=0.8\linewidth]{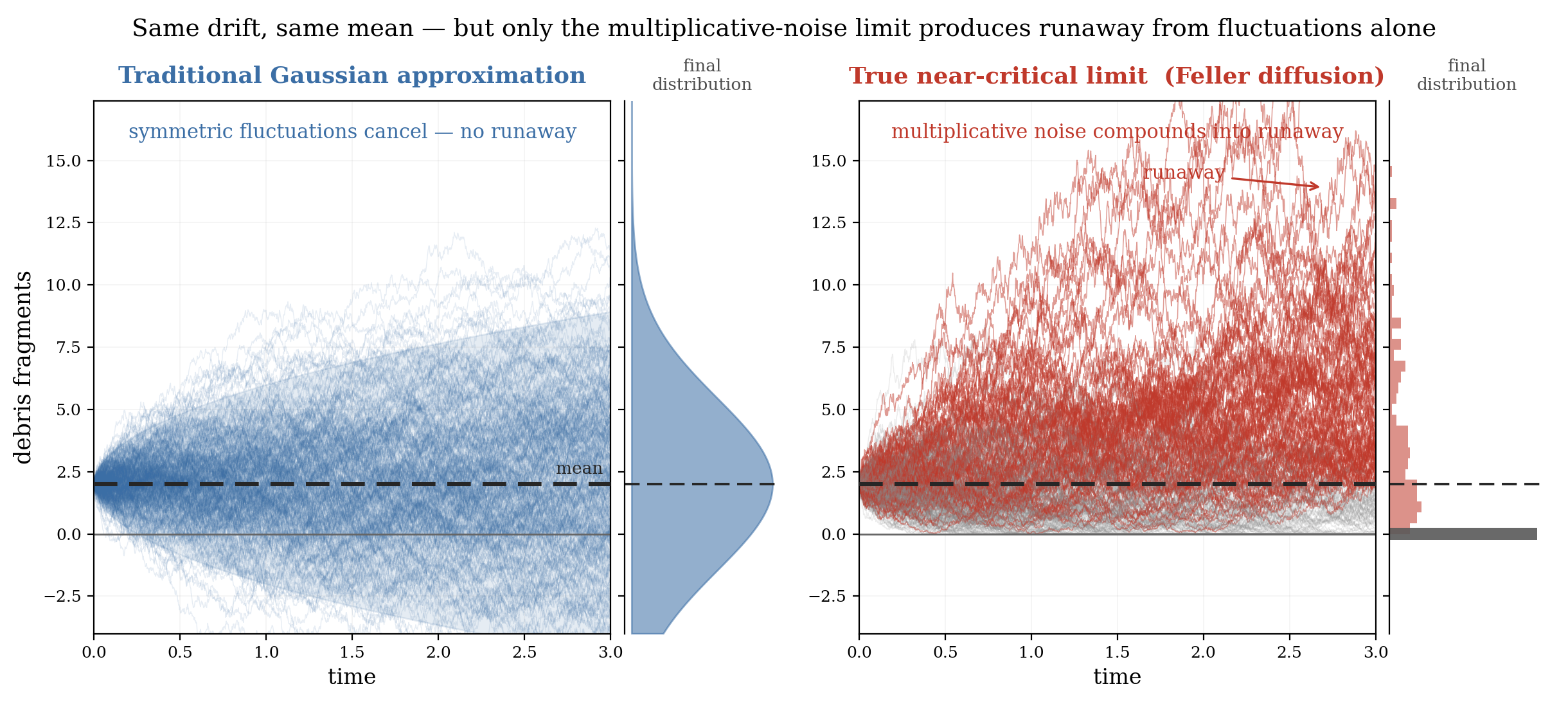}
    \caption{Comparison of Gaussian SDE and Feller Diffusion at the critical Kessler threshold. 
    For the Feller diffusion, the gray sample paths are those which are absorbed at $0$, and the red sample paths are the ones which escape to infinity.}
    \label{fig:gaussian-feller}
\end{figure}

\begin{figure}[t]
\centering
\includegraphics[width=\textwidth]{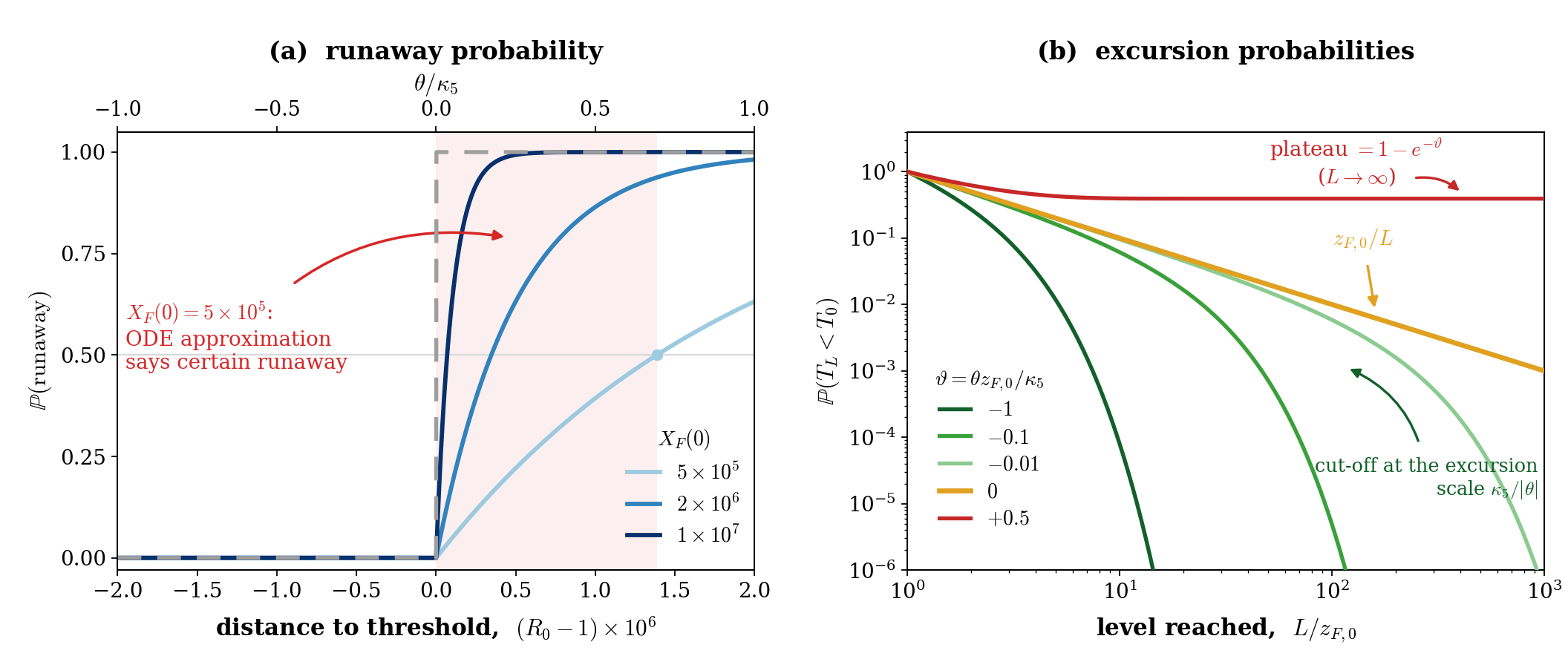}
\caption{(a) Stochastic Kessler runaway probability, against the distance
to the Kessler threshold, for three initial fragment counts with
$N^{\alpha_F}=5\times10^{5}$. The ODE approximation is the dashed step at exact criticality. The shaded region marks where the ODE approximation predicts certain runaway, whereas in the Feller diffusion, runaway is a random event.
(b) Excursion probabilities of the fragment counts in dimensionless form,
$\vartheta:=\theta z_{F,0}/\kappa_5$. At exact criticality the population
reaches any level $L$ with probability $z_{F,0}/L$ on zero drift; below the
threshold the same power law holds up to the excursion scale $\kappa_5/|\theta|$
           and then cuts off; above it the curve plateaus at the survival probability.}
\label{fig:runaway-probability}
\end{figure}


\section{ODE Approximations on the Fast and Intermediate Time-Scales}
\label{sec:fluid-limit}

We use the following notation here and throughout the remainder of this paper.

If $\alpha$ and $\nu$ are vectors, then $\alpha \cdot \nu$ denotes the inner product. Throughout, $z = (z_I, z_F)$ denotes a generic point of $\bR_{\geq 0}^2$, $\|\cdot\|$ is the Euclidean norm, $x \wedge y := \min(x, y)$, and $\probc$ denotes convergence in probability. 
We write $o_p(1)$ (respectively, \ $O_p(1)$) for a term that converges to $0$ in probability (respectively,\ is stochastically bounded), uniformly on $[0, T]$.
For functions $f, g$, we write $f \asymp g$ if there exist constants $C_1, C_2$ with $C_1 \leq \left|\frac{f}{g}\right| \leq C_2$, and $f \lesssim g$ if there exists a constant $C_3$ with $f \leq C_3 g$.
The function $\mathrm{id}$ is the identity function on $\bR$. If $(X_n)_{n \in \bN}$ and $X$ are right-continuous functions with left limits on the Skorokhod space $D_\bR[0, \infty)$, then $X_n \Rightarrow X$ denotes weak convergence in $D_\bR[0, \infty)$.
We also use $\Rightarrow$ to denote weak convergence of random variables.

\subsection{Auxiliary Lemmas}
Fix a clock $\gamma \geq 0$. Let $\zeta_{ki}$ be the $i$-component of the net-change vector $\zeta_k$, so that $(\zeta_{kI})_{k=1}^5 = (1, -2, -1, -1, 0)$ and $(\zeta_{kF})_{k=1}^5 = (0, 2, 1, 0, -1)$, and define the scaled mass-action monomials
\[
\phi_1 \equiv 1, \quad \phi_2(z) = z_I^2, \quad \phi_3(z) = z_I z_F, \quad \phi_4(z) = z_I, \quad \phi_5(z) = z_F .
\]
Using $\rho_k = \beta_k + \alpha\cdot\nu_k$, the rescaled dynamics above can be written, for $i \in \{I, F\}$, as
\[
Z_i^{N,\gamma}(t) = Z_i^N(0) + \sum_{k=1}^5 \zeta_{ki}\Big( N^{\theta_k^{(i)}}\kappa_k \int_0^t \phi_k\big(Z^{N,\gamma}(s)\big)\,ds + M_{ki}^{N,\gamma}(t) \Big) + o_p(1),
\qquad \theta_k^{(i)} := \gamma + \rho_k - \alpha_i,
\]
where the $o_p(1)$ absorbs the discrete correction $Z_I^N - N^{-\alpha_I}$ in $\phi_2$, and
\[
M_{ki}^{N,\gamma}(t) := N^{-\alpha_i}\,\widetilde Y_k\!\Big( N^{\gamma + \rho_k}\kappa_k \int_0^t \phi_k\big(Z^{N,\gamma}(s)\big)\,ds \Big),
\qquad \widetilde Y_k := Y_k - \mathrm{id},
\]
is a martingale with predictable quadratic variation
\[
\big\langle M_{ki}^{N,\gamma} \big\rangle_t = N^{\,\theta_k^{(i)} - \alpha_i}\,\kappa_k \int_0^t \phi_k\big(Z^{N,\gamma}(s)\big)\,ds .
\]
The exponent $\theta_k^{(i)}$ governs whether reaction $k$ contributes to the drift of species $i$ on the clock $\gamma$: the drift term is of order $N^{\theta_k^{(i)}}$ and the associated martingale of order $N^{(\theta_k^{(i)} - \alpha_i)/2}$.

\begin{lemma}\label{lem:poisson}
Let $Y$ be a unit-rate Poisson process and $\widetilde Y = Y - \mathrm{id}$. Let $a \in \bR$, $b \geq 0$, and let $w^N$ be a nondecreasing adapted process with $w^N(t) \leq W t$ on $[0, T]$. Then
\[
\E\Big[ \sup_{t \leq T} \big( N^{-a}\widetilde Y(N^{b} w^N(t)) \big)^2 \Big] \leq 4 W T\, N^{\,b - 2a}.
\]
In particular, if $2a > b$ then $\sup_{t \leq T} \big| N^{-a}\widetilde Y(N^{b} w^N(t)) \big| \probc 0$.
\end{lemma}
\noindent\emph{Proof.} The process $u \mapsto \widetilde Y(u)$ is a martingale with $\E[\widetilde Y(u)^2] = u$. By optional sampling, $t \mapsto \widetilde Y(N^{b} w^N(t))$ is a martingale, and Doob's $L^2$ maximal inequality gives
\[
\E\Big[ \sup_{t \leq T} \widetilde Y(N^{b} w^N(t))^2 \Big] \leq 4\,\E\big[ \widetilde Y(N^{b} w^N(T))^2 \big] = 4\,\E\big[ N^{b} w^N(T) \big] \leq 4 W T\, N^{b}.
\]
Multiplying by $N^{-2a}$ yields the bound; the convergence follows from Markov's inequality once $b - 2a < 0$. \hfill$\blacksquare$
\medskip

\begin{lemma}\label{lem:apriori}
Under (A1)--(A3), for $\gamma \in \{\gamma_I, \gamma_F\}$, and for every $T < \infty$, the family $\big\{ \sup_{0 \leq t \leq T} \|Z^{N,\gamma}(t)\| \big\}_N$ is tight.
\end{lemma}
\noindent\emph{Proof.} Fix $L > \|z_0\| + 1$ and let $\tau_L := \inf\{ t : \|Z^{N,\gamma}(t)\| > L \}$; on $[0, \tau_L]$ each $\phi_k$ is bounded by a constant $C_L$. 
We bound the two coordinates in turn. 
In the intact equation the only production is the state-independent arrival rate~\eqref{eq:rxn1} (with $\zeta_{1I} = 1$), while the dynamics~\eqref{eq:rxn2}--\eqref{eq:rxn4} only remove intacts.
Since these removals are nonpositive and $\theta_1^{(I)} = \theta_4^{(I)} =: \Theta$, the intact drift $b^N(z) := \sum_k \zeta_{kI} N^{\theta_k^{(I)}}\kappa_k \phi_k(z)$ satisfies
\[
b^N(z) \leq a - c\, z_I, \qquad a := N^{\Theta}\kappa_1, \quad c := N^{\Theta}\kappa_4, \quad \frac{a}{c} = z_I^\star .
\]
Writing $Z_I^{N,\gamma}(t) = Z_I^N(0) + \int_0^t b^N(Z^{N,\gamma})\,ds + \mathcal M_I(t)$ with $\mathcal M_I := \sum_k \zeta_{kI} M_{kI}^{N,\gamma}$, and applying the integrating factor $e^{ct}$,
\[
Z_I^{N,\gamma}(t) \;\leq\; \Big[ e^{-ct}Z_I^N(0) + z_I^\star\big(1 - e^{-ct}\big)\Big] \;+\; \mathcal N(t), \qquad \mathcal N(t) := e^{-ct}\!\int_0^t e^{cs}\,d\mathcal M_I(s).
\]
The bracketed (deterministic) part is bounded by $\max(Z_I^N(0), z_I^\star)$, uniformly in $N$: the scaled rates $a, c \asymp N^{\Theta}$ enter only through the ratio $a/c = z_I^\star$. 
For the noise, $\langle \mathcal M_I\rangle \asymp N^{\Theta - \alpha_I}(\kappa_1 + \kappa_4 Z_I^{N,\gamma})$ (via dynamics (1), (4)), so 
\[\E\,\mathcal N(t)^2 \lesssim N^{\Theta - \alpha_I}(\kappa_1 + \kappa_4 K_0)/(2c) = O(N^{-\alpha_I}),\]
where $K_0 := \max(\E Z_I^N(0), z_I^\star)$ is the uniform mean bound obtained by taking expectations above. With Doob's inequality this gives $\sup_{t \leq T \wedge \tau_L}|\mathcal N| = O_p(N^{-\alpha_I/2})$, hence
\[
\sup_{t \leq T \wedge \tau_L} Z_I^{N,\gamma} \;\leq\; \max\big(Z_I^N(0), z_I^\star\big) + O_p(N^{-\alpha_I/2})
\]
is tight; call a uniform high-probability bound $K_I$. 
Given $Z_I^{N,\gamma} \leq K_I$, the fragment production is affine in $Z_F$: dynamics~\eqref{eq:rxn3} contributes at rate $\leq \kappa_3 K_I Z_F^{N,\gamma}$ (with coefficient $O(1)$) and dynamics~\eqref{eq:rxn2} at rate $\leq N^{\theta_2^{(F)}}\kappa_2 K_I^2$ with $\theta_2^{(F)} < 0$ (with coefficient $o(1)$). 
Gronwall's inequality applied to $\E \sup_{t \leq T \wedge \tau_L} Z_F^{N,\gamma}$, again using Lemma~\ref{lem:poisson} for the martingale, gives a uniform bound $K_F$. Taking $L > 2(K_I + K_F)$ and applying Markov's inequality forces $\pr(\tau_L \leq T) < \eta$ for large $N$, which is the claim.
\hfill$\blacksquare$
\medskip

We first derive ODE approximations on the fast and intermediate time-scales.
The key result in this section is the identification of the critical Kessler threshold in Corollary~\ref{cor:kessler}.

For the remainder of this paper, we make the following standing assumptions:
\begin{itemize}
    \item[(A1)] \emph{(Initial data.)} $Z^N(0) \probc z_0 = (z_{I,0}, z_{F,0})$, where $z_0$ is an arbitrary point satisfying $z_{I, 0} > 0$ and $z_{F, 0} > 0$.
    \item[(A2)] \emph{(Rate constants.)} The rate constants $\kappa_k$, as well as the abundance exponents $\alpha = (\alpha_I, \alpha_F)$ and the reaction exponents $\rho_k = \beta_k + \alpha\cdot\nu_k$ are held fixed.
    \item[(A3)] \emph{(Balance.)} With the balance conditions above, $\rho_1 = \rho_4 = \max(\rho_1, \rho_2, \rho_3, \rho_4)$ and $\max(\rho_2, \rho_3) < \rho_1$ (intacts); $\rho_3 = \rho_5 = \max(\rho_2, \rho_3, \rho_5)$ and $\rho_2 < \rho_5$ (fragments); and $\alpha_F > \alpha_I > 0$.
\end{itemize}
Under (A3) the natural time-scales $\gamma_I = \alpha_I - \rho_1 \geq 0$ and $\gamma_F = \alpha_F - \rho_5 \geq 0$ satisfy
\[
\gamma_F - \gamma_I = (\alpha_F - \alpha_I) + (\rho_1 - \rho_3) > 0,
\]
so the fragment clock $\gamma_F$ is strictly slower than the intact clock $\gamma_I$. We write \[z_I^\star := \kappa_1/\kappa_4.\]

\begin{theorem}\label{prop:fast}
    Under (A1)--(A3), as $N \to \infty$,
    \[ \sup_{t \leq T} \big| Z_I^{N,\gamma_I}(t) - \bar z_I(t) \big| \probc 0
    \qquad\text{and}\qquad
    \sup_{t \leq T} \big| Z_F^{N,\gamma_I}(t) - z_{F,0} \big| \probc 0, \]
    where $\bar z_I$ is the unique solution of
    \[ \dot{\bar z}_I(t) = \kappa_1 - \kappa_4\, \bar z_I(t), \qquad \bar z_I(0) = z_{I,0}, \]
    namely $\bar z_I(t) = z_I^\star + (z_{I,0} - z_I^\star) e^{-\kappa_4 t}$.
    In particular, the fragments are frozen at their initial value on the intact clock.
\end{theorem}

\begin{proof}
On the time-scale $\gamma_I$ we have $\theta_k^{(I)} = \rho_k - \rho_1$, so by (A3) $\theta_1^{(I)} = \theta_4^{(I)} = 0$ while $\theta_2^{(I)}, \theta_3^{(I)} < 0$; and $\theta_k^{(I)} - \alpha_I \leq -\alpha_I < 0$ for every $k$. 
For the fragment coordinate, $\theta_k^{(F)} = \gamma_I + \rho_k - \alpha_F = (\alpha_I - \alpha_F) + (\rho_k - \rho_1)$.
Since $\alpha_I < \alpha_F$ and $\rho_k \leq \rho_1$ for $k \in \{2, 3, 5\}$ (using $\rho_5 = \rho_3 < \rho_1$), we get $\theta_k^{(F)} < 0$; thus $\theta_k^{(F)} - \alpha_F < 0$.

By Lemma~\ref{lem:poisson} every martingale $M_{ki}^{N,\gamma_I}$ tends to $0$ uniformly on $[0, T]$ in probability, and by Lemma~\ref{lem:apriori} the integrals $\int_0^t \phi_k(Z^{N,\gamma_I})\,ds$ are stochastically bounded, so each drift term with $\theta_k^{(i)} < 0$ tends to $0$ as well. The fragment equation therefore collapses to $Z_F^{N,\gamma_I}(t) = Z_F^N(0) + o_p(1)$ uniformly on $[0, T]$, and (A1) gives the second claim.

By the same reasoning, for the intacts, we only need to consider $k = 1, 4$, and we get:
\[
Z_I^{N,\gamma_I}(t) = Z_I^N(0) + \kappa_1 t - \kappa_4 \int_0^t Z_I^{N,\gamma_I}(s)\,ds + \varepsilon_N(t),
\]
where $\sup_{t \leq T} |\varepsilon_N(t)| \probc 0$ collects the vanishing drifts and martingales. 
Subtracting the integral form $\bar z_I(t) = z_{I,0} + \kappa_1 t - \kappa_4 \int_0^t \bar z_I$ and using (A1),
\[
\big| Z_I^{N,\gamma_I}(t) - \bar z_I(t) \big| \leq \delta_N + \kappa_4 \int_0^t \big| Z_I^{N,\gamma_I}(s) - \bar z_I(s) \big|\,ds,
\qquad \delta_N := |Z_I^N(0) - z_{I,0}| + \sup_{t \leq T}|\varepsilon_N(t)| \probc 0 .
\]
Gronwall's inequality gives $\sup_{t \leq T} | Z_I^{N,\gamma_I} - \bar z_I | \leq \delta_N e^{\kappa_4 T} \probc 0$.
\end{proof}

\begin{theorem}\label{prop:slow}
    Under (A1)--(A3), as $N \to \infty$,
    \[ \sup_{t \leq T} \Big| \int_0^t \big( Z_I^{N,\gamma_F}(s) - z_I^\star \big)\,ds \Big| \probc 0
    \qquad\text{and}\qquad
    \sup_{t \leq T} \big| Z_F^{N,\gamma_F}(t) - \bar z_F(t) \big| \probc 0, \]
    where $\bar z_F$ solves the equation
    \[ \dot{\bar z}_F(t) = \big( \kappa_3\, z_I^\star - \kappa_5 \big)\, \bar z_F(t), \qquad \bar z_F(0) = z_{F,0}. \]
    Moreover, $Z_I^{N,\gamma_F}(t) \probc z_I^\star$ for every $t > 0$.
\end{theorem}

\begin{proof}
We proceed in two steps: first we show that the intact variable converges to its steady state, then we prove convergence for the fragment variable.

\emph{Step 1 (fast intact variable).} On the time-scale $\gamma_F$ the intact exponents are $\theta_k^{(I)} = \gamma_F + \rho_k - \alpha_I$. Put $\iota := \gamma_F - \gamma_I = \gamma_F + \rho_1 - \alpha_I > 0$; then $\theta_1^{(I)} = \theta_4^{(I)} = \iota$, while $\theta_2^{(I)}, \theta_3^{(I)} < \iota$ by (A3). Collecting terms, the intact equation reads
\[
Z_I^{N,\gamma_F}(t) = Z_I^N(0) + N^{\iota}\!\int_0^t \big( \kappa_1 - \kappa_4 Z_I^{N,\gamma_F}(s) \big)\,ds + \int_0^t R_N(s)\,ds + \mathcal M_I^N(t),
\]
where
\[
R_N(s) := -2N^{\theta_2^{(I)}}\kappa_2\,\big(Z_I^{N,\gamma_F}(s)\big)^2 - N^{\theta_3^{(I)}}\kappa_3\, Z_I^{N,\gamma_F}(s)\, Z_F^{N,\gamma_F}(s)
\]
collects the dynamics~\eqref{eq:rxn2} and~\eqref{eq:rxn3}.
By Lemma~\ref{lem:apriori}, $\sup_{t\le T}|R_N| = O_p(N^{\iota'})$ with $\iota' := \max(\theta_2^{(I)}, \theta_3^{(I)}) < \iota$, and $\mathcal M_I^N = \sum_k \zeta_{kI} M_{kI}^{N,\gamma_F}$ has quadratic variation $O_p(N^{\iota - \alpha_I})$.
Recall that the dominant contribution comes from dynamics~\eqref{eq:rxn1} and~\eqref{eq:rxn4}. 
Re-arranging, we get:
\begin{equation}
    \kappa_4 \int_0^t \big( Z_I^{N,\gamma_F}(s) - z_I^\star \big)\,ds
    = N^{-\iota}\Big[ -\big( Z_I^{N,\gamma_F}(t) - Z_I^N(0) \big) + \int_0^t R_N(s)\,ds + \mathcal M_I^N(t) \Big].
    \label{eq:slow-ODE-intact}
\end{equation}
The three terms in the bracket are $O_p(1)$ (by Lemma~\ref{lem:apriori}), $O_p(N^{\iota'})$, and $O_p(N^{(\iota - \alpha_I)/2})$ respectively; all tend to zero when multiplied by $N^{-\iota}$ since $\iota' < \iota$ and $\alpha_I > 0$. 
This proves the first display. 
The pointwise statement $Z_I^{N,\gamma_F}(t) \probc z_I^\star$ for $t > 0$ then follows from the exponential stability of the fast flow $\dot\zeta = N^{\iota}(\kappa_1 - \kappa_4 \zeta)$ whose contraction rate $N^{\iota}\kappa_4 \to \infty$.
Here, $Z_I^{N, \gamma_F}$ is the dominant term for the right hand side in Equation~\eqref{eq:slow-ODE-intact}, and $\zeta$ is the dynamics on the left hand side.

\emph{Step 2 (intermediate fragment variable).} On the same time-scale the fragment exponents are $\theta_k^{(F)} = \rho_k - \rho_5$, so $\theta_3^{(F)} = \theta_5^{(F)} = 0$, $\theta_2^{(F)} < 0$, and $\theta_k^{(F)} - \alpha_F \leq -\alpha_F < 0$. 
By Lemma~\ref{lem:poisson} and Lemma~\ref{lem:apriori}, the drift due to dynamics~\eqref{eq:rxn2} and all fragment martingales vanish uniformly, leaving
\[
Z_F^{N,\gamma_F}(t) = Z_F^N(0) + \kappa_3 \int_0^t Z_I^{N,\gamma_F}(s) Z_F^{N,\gamma_F}(s)\,ds - \kappa_5 \int_0^t Z_F^{N,\gamma_F}(s)\,ds + o_p(1).
\]
Write $Z_I Z_F = z_I^\star Z_F + (Z_I - z_I^\star) Z_F$ and set $A_N(t) := \int_0^t (Z_I^{N,\gamma_F} - z_I^\star)\,ds$, so $\sup_{t\le T}|A_N| \probc 0$ by Step~1. Integration by parts gives
\[
\int_0^t (Z_I^{N,\gamma_F} - z_I^\star) Z_F^{N,\gamma_F}\,ds
= A_N(t) Z_F^{N,\gamma_F}(t) - \int_0^t A_N(s)\,dZ_F^{N,\gamma_F}(s).
\]
Since $Z_F^{N,\gamma_F}$ is stochastically bounded, with drift of bounded variation plus a martingale of quadratic variation $O_p(N^{-\alpha_F})$, both terms on the right are $O_p(\sup_{t\le T}|A_N|) \probc 0$. 
Hence
\[
Z_F^{N,\gamma_F}(t) = Z_F^N(0) + \big( \kappa_3 z_I^\star - \kappa_5 \big) \int_0^t Z_F^{N,\gamma_F}(s)\,ds + e_N(t),
\qquad \sup_{t \leq T}|e_N| \probc 0.
\]
Comparing with the integral form of the averaged ODE and applying Gronwall's inequality yields \[\sup_{t \leq T} | Z_F^{N,\gamma_F} - \bar z_F | \probc 0,\]
which establishes the result.
\end{proof}

\begin{corollary}[Critical Kessler Threshold]\label{cor:kessler}
For the $\gamma_F$ time-scale in Theorem~\ref{prop:slow}, the fragment count tends to infinity with time if and only if
\[
z_I^\star = \frac{\kappa_1}{\kappa_4} > \frac{\kappa_5}{\kappa_3},
\]

and fragment count tends to zero if and only if
\[z_I^\star < \frac{\kappa_5}{\kappa_3}.\]

\end{corollary}

\begin{remark}
    It is useful to consolidate the critical Kessler threshold in a dimensionless form. In particular, note that each fragment collides with intacts at rate $\kappa_3 z_I^\star$ and de-orbits at rate $\kappa_5$. Therefore, over its lifetime a single fragment produces 
    \[
        R_0 := \frac{\kappa_3 z_I^\star}{\kappa_5}
    \]
    new fragments.
    Thus, Corollary~\ref{cor:kessler} is precisely the statement that the \emph{reproduction number} $R_0 > 1$, and the fragment population grows precisely when each fragment more than replaces itself. Furthermore, the growth rate for the fluid approximation $\bar{z}_F$ can be expressed as $\Delta = \kappa_5 (R_0 - 1)$.
\end{remark}

\begin{remark}
    \label{rem:B-W}
    When $\alpha_I = \alpha_F = 1$, and $\beta_1 = \ldots = \beta_5 = 0$, we recover the differential equations in the style of Bradley and Wein~\cite{bradley2009space}:
    \begin{align*}
        \dot{i} &= \kappa_1 - 2\kappa_2 i^2 - \kappa_3 if - \kappa_4 i, 
        \\
        \dot{f} &= 2\kappa_2 i^2 + \kappa_3 if - \kappa_5 f,
    \end{align*}
    where $i,f$ correspond to intacts and fragments, respectively. Here, the term ``in the style of'' indicates that the form of the equations is the same, but the coefficients may differ.
\end{remark}

\begin{remark}
The hypothesis $\alpha_I, \alpha_F > 0$ places both species in the \emph{fluid} regime, where the surviving martingales vanish and the limits are deterministic. If instead an abundance exponent equals $0$, the quadratic variation at the balanced reactions is $O(1)$ rather than $o(1)$, and the limit on that clock retains its jumps: one obtains a pure-jump Markov process, whose generator is read off from the surviving reactions, rather than an ODE. The averaging identity of Theorem~\ref{prop:slow} then persists with $z_I^\star$ replaced by the mean of the fast intact process under its quasi-stationary law (see~\cite{kang2013separation}).
\end{remark}

\begin{remark}[Reversed abundances and the sign of $\iota$]
\label{rem:reversed}
Recall that $\alpha_F > \alpha_I$ by assumption.
In this remark, we discuss our results in further generality.

The qualitative behavior of the fluid limit is governed not by the ordering of $\alpha_I$ and $\alpha_F$ directly, but by the sign of
\[
\iota \;=\; \gamma_F - \gamma_I \;=\; (\alpha_F - \alpha_I) + (\rho_1 - \rho_3),
\]
in which the abundance gap $\alpha_F - \alpha_I$ competes with the (positive) rate gap $\rho_1 - \rho_3 > 0$ fixed by the balance condition $\beta_1 > \beta_3 + \alpha_I + \alpha_F$. When $\alpha_I > \alpha_F$, three regimes arise.

\emph{(i) Mild reversal, $0 < \alpha_I - \alpha_F < \rho_1 - \rho_3$ (so $\iota > 0$).} The proofs of Theorems~\ref{prop:fast}--\ref{prop:slow} use only $\iota > 0$ and $\alpha_I, \alpha_F > 0$. 
Hence, fragments remain the slow (averaged) species, the intacts relax to $z_I^\star = \kappa_1/\kappa_4$, and the averaged equation and Kessler threshold of Corollary~\ref{cor:kessler} are unchanged.

\emph{(ii) Critical case, $\alpha_I - \alpha_F = \rho_1 - \rho_3$ (so $\iota = 0$, $\gamma_I = \gamma_F$).} The time-scale separation collapses: both species have $O(1)$ drift on a single clock, and there is no averaging. The fluid limit is the coupled system
\[
\dot z_I = \kappa_1 - \kappa_4 z_I, \qquad \dot z_F = (\kappa_3 z_I - \kappa_5)\, z_F .
\]

\emph{(iii) Strong reversal, $\alpha_I - \alpha_F > \rho_1 - \rho_3$ (so $\iota < 0$).} Formally the roles swap --- intacts slow, fragments fast --- but the fast fragment subsystem has no stabilizing equilibrium. On the fast clock the intact-intact collision dynamics~\eqref{eq:rxn2} is subdominant, leaving the homogeneous, autocatalytic dynamics $\dot z_F = (\kappa_3 z_I - \kappa_5) z_F$: its quasi-steady state is $z_F \equiv 0$ when $\kappa_3 z_I < \kappa_5$ --- which drives the fragments below order $N^{\alpha_F}$ and contradicts the balance condition --- and does not exist when $\kappa_3 z_I > \kappa_5$. The averaging principle is therefore degenerate, and the balance condition is self-consistent only when $\iota \geq 0$. Structurally, the assumption $\alpha_F > \alpha_I$ places the stable, mean-reverting species (intacts, with constant launch immigration and linear de-orbit removal) on the fast clock, and the unstable, autocatalytic species (fragments) on the slow clock; this is precisely what yields a well-posed averaged limit and places the Kessler instability on an observable slow time-scale.
The arguments are similar to those used to prove Theorems~\ref{prop:fast} and~\ref{prop:slow}, and we omit the details for the sake of brevity.
\end{remark}


\section{SDE Approximations on the Fast and Intermediate Time-Scales}
\label{sec:diffusion-limit}

In this section, we derive Gaussian SDE approximations to study stochastic fluctuations around the ODE approximations of the last subsection.
Our major point is that these approximations here are insufficient to truly understand the inherent randomness in the collision dynamics; instead one needs to look at the Feller diffusion in the next section.
Nevertheless, we derive the Gaussian SDE approximations in order to make that point precise.

The Gaussian fluctuations around the fluid limits of Theorems~\ref{prop:fast} and~\ref{prop:slow} are in the spirit of the multiple time-scale central limit theorems of~\cite{kang2014central}. 
For $i = I, F$, the correct normalization is $N^{\alpha_i/2}$ because at a balanced reaction ($\theta_k^{(i)} = 0$) the martingale $M_{ki}^{N,\gamma}$ is of order $N^{-\alpha_i/2}$. 

We strengthen (A1) to

\smallskip
\noindent(A1') \emph{(Initial fluctuations.)} There exist random variables $v_0, u_0$ with \[N^{\alpha_I/2}\big(Z_I^N(0) - \bar z_I^N(0)\big) \Rightarrow v_0, \qquad N^{\alpha_F/2}\big(Z_F^N(0) - \bar z_F^N(0)\big) \Rightarrow u_0,\]
where $\bar z^N = (\bar z_I^N, \bar z_F^N)$ solves the finite-$N$ scaled ODE $\dot{\bar z}^N = F^N(\bar z^N)$.
Here, $F^N$ is the drift field of the time-change representation, with $\bar z^N(0) \to z_0$.
Recall $\bar z^N \to \bar z$ uniformly on compacts, so centering at $\bar z^N$ or at the limit $\bar z$ agree whenever the subdominant reactions are separated by more than $\alpha_i/2$, and the finite-$N$ centering removes the deterministic bias otherwise. 

\begin{theorem}\label{prop:clt-fast}
    Assume (A1)--(A3) and (A1'). Then
    \[
    V^N := N^{\alpha_I/2}\big(Z_I^{N,\gamma_I} - \bar z_I^N\big) \;\Rightarrow\; V,
    \]
    the unique solution of the SDE
    \[
    dV(t) = -\kappa_4\, V(t)\,dt + \sqrt{\kappa_1 + \kappa_4\,\bar z_I(t)}\;dB(t), \qquad V(0) = v_0 .
    \]
    In particular $V$ is Gaussian; its variance $v(t) = \operatorname{Var} V(t)$ solves $\dot v = -2\kappa_4 v + \kappa_1 + \kappa_4 \bar z_I$, so $v(t) \to z_I^\star$ as $t \to \infty$.
\end{theorem}
\begin{proof}
Let $F_I^N(z) = \kappa_1 - \kappa_4 z_I - 2N^{\theta_2^{(I)}}\kappa_2 z_I^2 - N^{\theta_3^{(I)}}\kappa_3 z_I z_F$ be the intact drift on the clock $\gamma_I$, with $\theta_2^{(I)}, \theta_3^{(I)} < 0$. Subtracting $\dot{\bar z}_I^N = F_I^N(\bar z^N)$ from the basic representation and multiplying by $N^{\alpha_I/2}$,
\[
V^N(t) = V^N(0) + \int_0^t N^{\alpha_I/2}\big[ F_I^N(Z^{N,\gamma_I}(s)) - F_I^N(\bar z^N(s)) \big]\,ds + W^N(t),
\qquad W^N := N^{\alpha_I/2}\sum_k \zeta_{kI} M_{kI}^{N,\gamma_I}.
\]
A first-order Taylor expansion gives $N^{\alpha_I/2}\big[F_I^N(Z^{N,\gamma_I}) - F_I^N(\bar z^N)\big] = \partial_{z_I}F_I^N(\bar z^N)\,V^N + \varepsilon^N$, where
\[
\partial_{z_I}F_I^N(\bar z^N) = -\kappa_4 - 4N^{\theta_2^{(I)}}\kappa_2 \bar z_I^N - N^{\theta_3^{(I)}}\kappa_3 \bar z_F^N \longrightarrow -\kappa_4,
\]
and $\sup_{t \leq T}|\varepsilon^N| \probc 0$: it gathers the vanishing coupling to the (intermediate) fragment fluctuation, carrying the prefactor $N^{\theta_3^{(I)}}\kappa_3\bar z_I^N \to 0$, and the quadratic Taylor remainder is $O(N^{-\alpha_I/2}(V^N)^2)$ times bounded second derivatives, controlled by Lemma~\ref{lem:apriori}. 
Hence
\[
V^N(t) = V^N(0) - \kappa_4 \int_0^t V^N(s)\,ds + W^N(t) + \widetilde\varepsilon^N(t), \qquad \sup_{t \leq T}|\widetilde\varepsilon^N| \probc 0 .
\]
The martingale $W^N$ has
\[
\langle W^N \rangle_t = N^{\alpha_I}\sum_k \zeta_{kI}^2 \langle M_{kI}^{N,\gamma_I}\rangle_t = \sum_k \zeta_{kI}^2 N^{\theta_k^{(I)}}\kappa_k \int_0^t \phi_k(Z^{N,\gamma_I})\,ds \;\probc\; \int_0^t \big(\kappa_1 + \kappa_4 \bar z_I(s)\big)\,ds,
\]
because only dynamics~\eqref{eq:rxn1} and~\eqref{eq:rxn4} (with $\theta_k^{(I)} = 0$, $\zeta_{kI}^2 = 1$, $\phi_1 = 1$, $\phi_4 \to \bar z_I$) survive.
Dynamics~\eqref{eq:rxn2} and~\eqref{eq:rxn3} have $N^{\theta_k^{(I)}} \to 0$; the jumps of $W^N$ have size $N^{\alpha_I/2}\cdot N^{-\alpha_I} = N^{-\alpha_I/2} \to 0$. 
The martingale CLT yields $W^N \Rightarrow \int_0^\cdot \sqrt{\kappa_1 + \kappa_4 \bar z_I}\,dB$. Since the map sending $(V^N(0), W^N, \widetilde\varepsilon^N)$ to the solution of the linear equation is continuous (via Gronwall's inequality), the continuous-mapping theorem gives $V^N \Rightarrow V$.
The variance ODE is the standard computation for a linear SDE, and $\bar z_I(t) \to z_I^\star$ gives $v(\infty) = (\kappa_1 + \kappa_4 z_I^\star)/(2\kappa_4) = z_I^\star$. \end{proof}

\begin{theorem}\label{prop:clt-slow}
    Assume (A1)--(A3) and (A1'), together with the strict gap $\rho_1 > \rho_3$. Then
    \[
    U^N := N^{\alpha_F/2}\big(Z_F^{N,\gamma_F} - \bar z_F^N\big) \;\Rightarrow\; U,
    \]
    the unique solution of the linear SDE
    \[
    dU(t) = \big(\kappa_3 z_I^\star - \kappa_5\big) U(t)\,dt + \sqrt{\big(\kappa_3 z_I^\star + \kappa_5\big)\bar z_F(t)}\;dB(t), \qquad U(0) = u_0 .
    \]
\end{theorem}

\begin{proof}
Centering at $\bar z_F^N$ and multiplying by $N^{\alpha_F/2}$, only the fragment reactions $k = 3, 5$ survive (Theorem~\ref{prop:slow}), so
\[
U^N(t) = U^N(0) + \int_0^t N^{\alpha_F/2}\big[ \kappa_3 Z_I^{N,\gamma_F}Z_F^{N,\gamma_F} - \kappa_5 Z_F^{N,\gamma_F} - (\kappa_3 z_I^\star - \kappa_5)\bar z_F^N \big]\,ds + \mathcal W_F^N(t) + o_p(1),
\]
with $\mathcal W_F^N := N^{\alpha_F/2}\big(\zeta_{3F} M_{3F}^{N,\gamma_F} + \zeta_{5F} M_{5F}^{N,\gamma_F}\big)$. 
Using $\kappa_3 Z_I Z_F - \kappa_3 z_I^\star \bar z_F = \kappa_3 z_I^\star (Z_F - \bar z_F) + \kappa_3 (Z_I - z_I^\star) Z_F$ and $\bar z_F^N \to \bar z_F$,
\[
U^N(t) = U^N(0) + \big(\kappa_3 z_I^\star - \kappa_5\big)\!\int_0^t U^N(s)\,ds + \Xi^N(t) + \mathcal W_F^N(t) + o_p(1),\]
\[\Xi^N(t) := \kappa_3 N^{\alpha_F/2}\!\int_0^t (Z_I^{N,\gamma_F} - z_I^\star) Z_F^{N,\gamma_F}\,ds .
\]

As in Theorem~\ref{prop:clt-fast}, $\langle \mathcal W_F^N\rangle_t \probc \int_0^t (\kappa_3 z_I^\star + \kappa_5)\bar z_F\,ds$ (from $k = 3, 5$, $\zeta_{kF}^2 = 1$, $\phi_3 \to z_I^\star \bar z_F$, $\phi_5 \to \bar z_F$), with jumps of size $N^{-\alpha_F/2} \to 0$; hence $\mathcal W_F^N \Rightarrow \int_0^\cdot \sqrt{(\kappa_3 z_I^\star + \kappa_5)\bar z_F}\,dB$.

Let $\xi^N := N^{\alpha_I/2}(Z_I^{N,\gamma_F} - z_I^\star)$. By Step~1 of Theorem~\ref{prop:slow} it is a fast Ornstein--Uhlenbeck process,
\[
d\xi^N(t) = -N^{\iota}\kappa_4\, \xi^N(t)\,dt + dW_\xi^N(t), \qquad \tfrac{d}{dt}\langle W_\xi^N\rangle_t \asymp N^{\iota}\big(\kappa_1 + \kappa_4 z_I^\star\big) = 2\kappa_1 N^{\iota},
\]
with $\iota = \gamma_F - \gamma_I > 0$, whose stationary variance is $z_I^\star$, so $\|\xi^N\|_{\infty,[0,T]} = O_p(1)$. From $N^{\iota}\kappa_4\int_0^t \xi^N\,ds = W_\xi^N(t) - (\xi^N(t) - \xi^N(0))$ we obtain $\operatorname{Var}\big(\int_0^t \xi^N\,ds\big) \sim \tfrac{2\kappa_1}{\kappa_4^2} N^{-\iota} t$. Since $Z_F^{N,\gamma_F} = \bar z_F + O_p(N^{-\alpha_F/2})$ and, by integration by parts as in Theorem~\ref{prop:slow}, $\Xi^N = \kappa_3 N^{(\alpha_F - \alpha_I)/2}\int_0^t \xi^N \bar z_F\,ds + o_p(1)$,
\[
\operatorname{Var} \Xi^N(t) \asymp N^{\alpha_F - \alpha_I}\cdot N^{-\iota} = N^{-(\rho_1 - \rho_3)} \longrightarrow 0,
\qquad \iota = (\alpha_F - \alpha_I) + (\rho_1 - \rho_3).
\]
Hence $\Xi^N \probc 0$ uniformly on $[0, T]$. Consequently $U^N$ solves a linear equation driven by $U^N(0) \Rightarrow u_0$ and $\mathcal W_F^N$ with $o_p(1)$ perturbations, and Gronwall's inequality and the continuous mapping theorem give $U^N \Rightarrow U$. 
\end{proof}

\begin{remark}
    \label{rem:additive-GSDE-noise}
    Observe that the noise in Theorem~\ref{prop:clt-slow} is additive, and the solution of the SDE is $U(t) = e^{\Delta t} (u_0 + M(t))$, where $M$ is the Gaussian martingale $M(t) = \int_0^t e^{-\Delta s}\sqrt{\kappa_5(R_0 + 1) \bar z_F(s)} dB(s)$. This gives the formal approximation
    \[
        Z_F^{N,\gamma_F}(t) \approx e^{\Delta t} \left[ z_{F,0} + N^{-\alpha_F/2} \left(u_0 + M(t) \right)\right].
    \]
    For $R_0 > 1$, the fluctuations remain subordinate to the fluid term $e^{\Delta t} z_{F,0}$ by the fixed factor $N^{-\alpha_F/2}$, uniformly in $t$. The bracketed term is non-positive with probability $\exp(-\Theta(N^{\alpha_F}))$. Therefore, extinction is a large deviation event on this scale. On the other hand, at $R_0 =1$, we have $\text{Var}(M(t)) = 2 \kappa_5 z_{F,0} t$ and the stochastic fluctuation term reaches the size of the fluid term at time horizon of order $N^{\alpha_F}$. Thus, criticality is the only regime in which this horizon is a power of $N$. This is the near-critical family of (A2') in the next section on the time-scale $\gamma_F+\alpha_F$.
\end{remark}

\begin{remark}[The critical case $\rho_1 = \rho_3$]
\label{rem:critical-GSDE}
When the dominant intact reaction and the intact--fragment collision balance at the \emph{same} order, $\rho_1 = \rho_3$, one has $\iota = \alpha_F - \alpha_I$ and the fast-fluctuation term $\Xi^N$ no longer vanishes: it is $O_p(1)$. Homogenizing the fast Ornstein--Uhlenbeck noise through $N^{\iota}\kappa_4 \int_0^t \xi^N\,ds = W_\xi^N(t) - (\xi^N(t) - \xi^N(0))$ gives
\[
\Xi^N \;\Rightarrow\; \frac{\kappa_3}{\kappa_4}\int_0^t \bar z_F(s)\,dG(s), \qquad \langle G \rangle_t = 2\kappa_1\!\int_0^t \bar z_F(s)^2\,ds,
\]
where $G$ is a Brownian motion produced by the invariance principle for the fast process, asymptotically independent of the intrinsic fragment noise (dynamics $1, 4$ being independent of dynamics $3, 5$ up to leading order). 
The two Gaussian sources add, so $U^N \Rightarrow U$ now solves
\[
dU = \big(\kappa_3 z_I^\star - \kappa_5\big) U\,dt + \sqrt{\big(\kappa_3 z_I^\star + \kappa_5\big)\bar z_F(t) + \frac{2\kappa_1 \kappa_3^2}{\kappa_4^2}\,\bar z_F(t)^2}\;dB .
\]
The extra term $\tfrac{2\kappa_1 \kappa_3^2}{\kappa_4^2}\bar z_F^2$ is the effective diffusivity contributed by the averaged-out fast dynamics --- the diffusion-approximation phenomenon in~\cite{kang2014central}. 
\end{remark}


\section{Slow Time-Scale Limit at the Critical Kessler Threshold}
\label{sec:critical}

The ODE and SDE limits above were derived at a \emph{fixed} distance from the critical Kessler threshold, measured by the growth rate $\Delta := \kappa_3 z_I^\star - \kappa_5$ of Corollary~\ref{cor:kessler}, with the rate constants held fixed. We now consider instead a \emph{near-critical family}, indexed by the system size $N$: the rate constants form convergent sequences $\kappa_k^{(N)} \to \kappa_k \in (0, \infty)$ (see~(A2') below), tuned so that the intact steady state $z_I^{\star,N} := \kappa_1^{(N)}/\kappa_4^{(N)}$ approaches the critical intact abundance,
\[
z_I^{\star,N} \;=\; \frac{\kappa_1^{(N)}}{\kappa_4^{(N)}} \;\longrightarrow\; z_I^{\mathrm{crit}} := \frac{\kappa_5}{\kappa_3},
\qquad
\Delta_N \;:=\; \kappa_3^{(N)}\, z_I^{\star,N} - \kappa_5^{(N)} \;\longrightarrow\; 0 .
\]
Here and below, $\kappa_k$ denotes the limit $\lim_N \kappa_k^{(N)}$, so that $z_I^{\mathrm{crit}} = \kappa_5/\kappa_3$. Thus $\Delta_N$ is the $N$-dependent analogue of the fixed growth rate $\Delta$: it measures the (vanishing) distance of the $N$th system from criticality, and the sign of $\Delta$ plays the same role in determining the asymptotic (in time) behavior of the system as in Bradley and Wein~\cite{bradley2009space}. The near-criticality may be driven by tuning any of the four rate constants $\kappa_1^{(N)}, \kappa_3^{(N)}, \kappa_4^{(N)}, \kappa_5^{(N)}$, not only launch and de-orbit.
At exact criticality the averaged fragment equation $\dot{\bar z}_F = \Delta\, \bar z_F$ degenerates to $\dot{\bar z}_F = 0$: the deterministic fragment dynamics vanishes, and the fate of the debris population is decided entirely by fluctuations.

The structural reason is that, with the intacts averaged at $z_I^{\mathrm{crit}}$, each fragment is replicated through the collision~\eqref{eq:rxn3} at per-capita rate $\kappa_3 z_I^{\mathrm{crit}}$ and removed through~\eqref{eq:rxn5} at rate $\kappa_5$. 
The Kessler threshold $\kappa_3 z_I^{\mathrm{crit}} = \kappa_5$ is precisely the critical branching condition, at which the natural scaling limit is a \emph{Feller diffusion}, rather than a Gaussian process.



\smallskip
\noindent(A2') \emph{(Near-critical family.)} The rate constants form convergent sequences $\kappa_k^{(N)} \to \kappa_k \in (0, \infty)$ with $z_I^{\star,N} \to z_I^{\mathrm{crit}}$ and $\theta := \lim_{N} N^{\alpha_F}\Delta_N \in \bR$, and the exponents $\alpha, \rho$ and the balance (A3) hold as before. In addition, \emph{both} feedbacks of the subdominant collisions are negligible on the branching scale:
\[
\rho_1 - \rho_3 > \alpha_F \quad (\text{intact depletion by~\eqref{eq:rxn3}}), \qquad \rho_5 - \rho_2 > \alpha_F \quad (\text{fragment immigration by~\eqref{eq:rxn2}}).
\]
The first strengthens the gap $\rho_1 > \rho_3$ of (A3).

On the
intermediate time-scale $\gamma_F$, the drift displaces the fragment density by an $O(1)$
amount over a horizon $1/\Delta_N$, while the fluctuations of
Theorem~\ref{prop:clt-slow}, of size $N^{-\alpha_F/2}$, require a horizon
$N^{\alpha_F}$ to accumulate to the same order. Away from the threshold the
first horizon is far shorter and the dynamics are deterministic to
leading order, as Theorems~\ref{prop:clt-fast} and~\ref{prop:clt-slow}
show. The two coincide when $\Delta_N \asymp N^{-\alpha_F}$, and this fixes
the near-critical family (A2') and its parameter
$\theta = \lim_N N^{\alpha_F}\Delta_N$, and simultaneously identifies
$N^{\alpha_F}$ as the horizon on which the limit must be taken. That is, the time-scale $\gamma_F + \alpha_F$.

\begin{theorem}[Near-critical convergence to Feller diffusion]\label{thm:critical}
    Under (A1), (A2'), and (A3), define the branching-scale fragment process
    \[
    \mathcal Z_F^N(\tau) := Z_F^{N,\gamma_F}\big(N^{\alpha_F}\tau\big).
    \]
    Then, as $N \to \infty$, $\mathcal Z_F^N \Rightarrow z_F$, where $z_F$ is the unique strong solution of the Feller diffusion with drift
    \begin{align}\label{eq:feller}
       dz_F(\tau) = \theta\, z_F(\tau)\,d\tau + \sqrt{2\kappa_5\, z_F(\tau)}\;dB(\tau), \qquad z_F(0) = z_{F,0}, 
    \end{align}
    absorbed at $0$.
\end{theorem}

\begin{proof}
Throughout, the rate constants are the $N$-dependent $\kappa_k^{(N)}$ of the near-critical family (A2'), with limits $\kappa_k^{(N)} \to \kappa_k$, and the intact steady state is $z_I^{\star,N} := \kappa_1^{(N)}/\kappa_4^{(N)} \to z_I^{\mathrm{crit}}$; we write the finite-$N$ quantities with the superscript $(N)$ and their limits without, and all limits are as $N \to \infty$.

On the $\gamma_F$ time-scale, the fragment dynamics~\eqref{eq:rxn3},~\eqref{eq:rxn5} fire at rate $N^{\gamma_F + \rho_3} = N^{\alpha_F}$; we pass to branching time $t = N^{\alpha_F}\tau$. Splitting the birth--death drift as $\kappa_3^{(N)} Z_I Z_F - \kappa_5^{(N)} Z_F = \Delta_N Z_F + \kappa_3^{(N)}(Z_I - z_I^{\star,N}) Z_F$ and using $N^{\alpha_F}\!\int_0^\tau f(N^{\alpha_F}\sigma)\,d\sigma = \int_0^{N^{\alpha_F}\tau} f(s)\,ds$, the fragment coordinate of the basic representation becomes
\begin{equation}\label{eq:crit-decomp}
\mathcal Z_F^N(\tau) = z_{F,0} + \int_0^\tau N^{\alpha_F}\Delta_N\, \mathcal Z_F^N\,d\sigma + \mathcal E_1^N(\tau) + \mathcal E_2^N(\tau) + \mathcal M^N(\tau),
\end{equation}
where, writing $Z_I = Z_I^{N,\gamma_F}$, $Z_F = Z_F^{N,\gamma_F}$, and the fragment martingale $M_F^N := \sum_{k \in \{2,3,5\}}\zeta_{kF} M_{kF}^{N,\gamma_F}$,
\[
\mathcal E_1^N(\tau) := \kappa_3^{(N)}\!\int_0^{N^{\alpha_F}\tau} (Z_I - z_I^{\star,N}) Z_F\,ds, \qquad
\mathcal E_2^N(\tau) := 2\kappa_2^{(N)} N^{\alpha_F + \theta_2^{(F)}}\!\int_0^\tau (\mathcal Z_I^N)^2\,d\sigma, \qquad
\mathcal M^N(\tau) := M_F^N(N^{\alpha_F}\tau) .
\]

\emph{Step 1: main drift and recovery of the noise.} By assumption, we have $N^{\alpha_F}\Delta_N \to \theta$. 
The martingale is where the abundance suppression $N^{-\alpha_F}$ is canceled by the length $N^{\alpha_F}$ of the branching horizon.
Since $M_F^N$ has quadratic-variation rate $N^{-\alpha_F}\big[(\kappa_3^{(N)} z_I^{\star,N} + \kappa_5^{(N)}) Z_F + O(N^{\theta_2^{(F)}})\big]$ per unit $\gamma_F$-time (dynamics~\eqref{eq:rxn3},~\eqref{eq:rxn5}, $\zeta_{kF}^2 = 1$, $\theta_k^{(F)} = 0$), the substitution $s = N^{\alpha_F}\sigma$ gives
\begin{equation}\label{eq:crit-qv}
\langle \mathcal M^N\rangle_\tau
= N^{-\alpha_F}\cdot N^{\alpha_F}\!\int_0^\tau \big(\kappa_3^{(N)} z_I^{\star,N} + \kappa_5^{(N)}\big)\, \mathcal Z_F^N\,d\sigma + o_p(1)
\;\longrightarrow\; \int_0^\tau 2\kappa_5\, z_F\,d\sigma,
\end{equation}
where we also use $\kappa_3^{(N)} z_I^{\star,N} \to \kappa_5$ and $\kappa_5^{(N)} \to \kappa_5$. On the fixed horizon of Theorem~\ref{prop:slow} the same quantity is $O(N^{-\alpha_F}) \to 0$ --- which is exactly why that limit is deterministic.
The branching horizon is the scale, $N^{\alpha_F}$ times longer, on which the accumulated quadratic variation is restored to $O(1)$. 

\emph{Step 2: localization and intact estimates.} Fix $L > z_{F,0}$ and set $\tau_L := \inf\{\tau : \mathcal Z_F^N(\tau) > L\}$. All estimates below hold on $[0, \tau_L]$, where $\mathcal Z_F^N \leq L$, with constants that may depend on $L$; we take $N \to \infty$ first and then $L \to \infty$.
On the $\gamma_F$ time-scale, the intact fluctuation $\eta^N := Z_I - z_I^{\star,N}$ obeys
\[
d\eta^N = -N^{\iota}\kappa_4^{(N)}\,\eta^N\,ds + R_N\,ds + dM_I^N, \qquad R_N = -N^{\theta_3^{(I)}}\kappa_3^{(N)} Z_I Z_F - 2N^{\theta_2^{(I)}}\kappa_2^{(N)} Z_I^2 \leq 0,
\]
with intact martingale $M_I^N$ of rate $d\langle M_I^N\rangle \asymp N^{\iota - \alpha_I}(\kappa_1^{(N)} + \kappa_4^{(N)} Z_I)\,ds$. The integrating factor $e^{-N^{\iota}\kappa_4^{(N)} s}$ gives $\sup_s \E[Z_I(s)^2 \ind_{\{s \leq \tau_L\}}] \leq C$ and the fast Ornstein--Uhlenbeck (stationary-variance) bound
\begin{equation}\label{eq:crit-eta}
\sup_{0 \leq s \leq N^{\alpha_F}T} \E\big[(\eta^N(s))^2\, \ind_{\{s \leq \tau_L\}}\big] = O(N^{-\alpha_I}),
\end{equation}
with the driving noise contributing $N^{\iota - \alpha_I}/(2N^{\iota}\kappa_4^{(N)}) = O(N^{-\alpha_I})$ and the subdominant drift contributing $\|R_N/(N^{\iota}\kappa_4^{(N)})\|^2 = O_p(N^{2(\theta_3^{(I)} - \iota)}) = O_p(N^{-2(\rho_1 - \rho_3)})$.

\emph{Step 3: the intact feedback vanishes.} Solving the intact equation for \[\eta^N\,ds = (N^{\iota}\kappa_4^{(N)})^{-1}\big(-d\eta^N + R_N\,ds + dM_I^N\big)\] and integrating by parts in the first term,
\[
\mathcal E_1^N(\tau)
= \frac{\kappa_3^{(N)}}{N^{\iota}\kappa_4^{(N)}}\bigg[\Big(\!-Z_F\eta^N\big|_0^{N^{\alpha_F}\tau} + \!\int_0^{N^{\alpha_F}\tau}\!\eta^N\,dZ_F\Big) + \int_0^{N^{\alpha_F}\tau}\! Z_F R_N\,ds + \int_0^{N^{\alpha_F}\tau}\! Z_F\,dM_I^N\bigg].
\]
Using $Z_F \leq L$, the second-moment bound~\eqref{eq:crit-eta}, and $\theta_3^{(I)} - \iota = \rho_3 - \rho_1$, we bound the three groups (in $L^2$, after the prefactor $N^{-\iota}$):
\begin{itemize}
    \item[\rm(i)] the boundary term is $O_p(N^{-\iota - \alpha_I/2})$ by Cauchy--Schwarz; $\int \eta^N\,dZ_F$ has drift part $\asymp \kappa_3^{(N)}\!\int(\eta^N)^2 Z_F\,ds$, with $\E \leq \kappa_3^{(N)} L\!\int \E[(\eta^N)^2]\,ds = O(L\, N^{\alpha_F - \alpha_I})$, and martingale part $\int \eta^N\,dM_F^N$ of quadratic variation $\int (\eta^N)^2\,d\langle M_F^N\rangle$, $\E \asymp N^{-\alpha_I}$; after $N^{-\iota}$, group~(i) is $O_p(N^{\alpha_F - \alpha_I - \iota}) = O_p(N^{-(\rho_1 - \rho_3)})$;
    \item[\rm(ii)] the dynamics \eqref{eq:rxn3} part of $\int Z_F R_N\,ds$ equals $-\,N^{\alpha_F - (\rho_1 - \rho_3)}\tfrac{\kappa_3^2 z_I^{\mathrm{crit}}}{\kappa_4}\!\int_0^\tau(\mathcal Z_F^N)^2\,d\sigma + o_p(1)$ (the \emph{logistic} term), and the dynamics \eqref{eq:rxn2} part is $O_p(N^{\alpha_F - (\rho_1 - \rho_2)})$; both vanish under (A2'), since $\rho_1 - \rho_2 > \rho_1 - \rho_3 > \alpha_F$;
    \item[\rm(iii)] $\int Z_F\,dM_I^N$ is a martingale of quadratic variation $\int Z_F^2\,d\langle M_I^N\rangle \leq L^2\!\int d\langle M_I^N\rangle \asymp L^2 N^{\iota - \alpha_I + \alpha_F}$, hence of size $O_p(L\, N^{(\iota - \alpha_I + \alpha_F)/2})$; and after multiplying by $N^{-\iota}$, it is $O_p(L\, N^{-(\rho_1 - \rho_3)/2})$.
\end{itemize}
The slowest term is~(iii), so $\sup_{\tau \leq T \wedge \tau_L}|\mathcal E_1^N| = O_p(N^{-(\rho_1 - \rho_3)/2}) \probc 0$.
Likewise, by~\eqref{eq:crit-eta} and $Z_I \leq z_I^{\star,N} + |\eta^N|$, we have \[\E|\mathcal E_2^N| \leq 2\kappa_2^{(N)} N^{\alpha_F + \theta_2^{(F)}}\!\int_0^\tau \E[(\mathcal Z_I^N)^2]\,d\sigma = O(N^{\alpha_F - (\rho_5 - \rho_2)}) \to 0,\] so $\mathcal E_2^N \probc 0$.

\emph{Step 4: moment bound, non-explosion, and tightness.} Insert Steps 1 and 3 into~\eqref{eq:crit-decomp}: on $[0, \tau_L \wedge T]$,
\[
\mathcal Z_F^N(\tau) = z_{F,0} + \int_0^\tau (\theta + o(1))\, \mathcal Z_F^N\,d\sigma + r_N(\tau) + \mathcal M^N(\tau), \qquad \E\sup_{\tau \leq T \wedge \tau_L}|r_N| \to 0 .
\]
Set $m_N(\tau) := \E\sup_{\tau' \leq \tau \wedge \tau_L}\mathcal Z_F^N$. By~\eqref{eq:crit-qv} and Doob's inequality, $\E\sup_{\tau' \leq \tau \wedge \tau_L}|\mathcal M^N| \leq 2\big(\E\langle \mathcal M^N\rangle_{\tau \wedge \tau_L}\big)^{1/2} \leq C\big(\int_0^\tau m_N\big)^{1/2} \leq \tfrac{C^2}{2} + \tfrac12\!\int_0^\tau m_N$, so
\[
m_N(\tau) \leq z_{F,0} + o(1) + \tfrac{C^2}{2} + \big(|\theta| + \tfrac12\big)\!\int_0^\tau m_N\,d\sigma ,
\]
and Gronwall's inequality gives $m_N(\tau) \leq C_T$ on $[0, T]$, uniformly in large $N$ and in $L$. Hence $\pr(\tau_L \leq T) \leq \pr\big(\sup_{\tau \leq T \wedge \tau_L}\mathcal Z_F^N \geq L\big) \leq m_N(T)/L \leq C_T/L$, so $\lim_{L \to \infty}\limsup_N \pr(\tau_L \leq T) = 0$: the localization is asymptotically negligible, $\{\sup_{\tau \leq T}\mathcal Z_F^N\}_N$ is tight, and $\mathcal Z_F^N$ does not explode. With the quadratic-variation control~\eqref{eq:crit-qv}, the family $\{\mathcal Z_F^N\}$ is tight in $D_{\bR}[0,\infty)$ by the Aldous--Rebolledo criterion~\cite{ethierkurtz1986}.

\emph{Step 5: identification.} By Steps 1--3, every weak limit point $z_F$ of $\{\mathcal Z_F^N\}$ solves the martingale problem for the generator $\mathcal A f(z) = \theta z f'(z) + \kappa_5 z f''(z)$: the drift is $\theta z_F$ (as $N^{\alpha_F}\Delta_N \to \theta$ and $\mathcal E_1^N, \mathcal E_2^N \probc 0$) and the quadratic variation is $\int_0^\cdot 2\kappa_5 z_F\,d\sigma$ by~\eqref{eq:crit-qv}. Since the diffusion coefficient $\sqrt{2\kappa_5 z}$ is H\"older-$\tfrac12$ and degenerates at $0$, pathwise uniqueness holds by the Yamada--Watanabe criterion~\cite{yamada1971}, so this martingale problem is well posed. Therefore the whole sequence converges, $\mathcal Z_F^N \Rightarrow z_F$, the unique Feller diffusion $dz_F = \theta z_F\,d\tau + \sqrt{2\kappa_5 z_F}\,dB$ absorbed at $0$. 
\end{proof}

The parameter $\theta = \lim_N N^{\alpha_F}\Delta_N$ is the rescaled distance to the Kessler threshold, and it replaces the deterministic sign test of Corollary~\ref{cor:kessler} by a probabilistic one. 

\begin{corollary}[Excursion probabilities]\label{cor:exit}
Let $T_x := \inf\{\tau : z_F(\tau) = x\}$. For $0 < z_{F,0} < L$,
\[
  \pr_{z_{F,0}}\bigl(T_L < T_0\bigr)
  \;=\; \frac{1 - e^{-\theta z_{F,0}/\kappa_5}}{1 - e^{-\theta L/\kappa_5}},
\]
which at exact criticality $(\theta = 0)$ reduces to $z_{F,0}/L$.
\end{corollary}

\begin{proof}
The scale density of \eqref{eq:feller} is $s'(x) = e^{-\theta x/\kappa_5}$,
and the exit probability is $\bigl(s(z_{F,0}) - s(0)\bigr) / \bigl(s(L) -
s(0)\bigr)$; see \cite[Ch.~15]{ethierkurtz1986}. The critical case follows by
letting $\theta \to 0$.
\end{proof}

Corollary~\ref{cor:exit} is the precise sense in which fluctuations alone
drive the fragment population to arbitrary levels. At exact criticality the
drift vanishes identically, yet the population reaches any prescribed level
$L$ with probability $z_{F,0}/L$ before clearing. Thus, the excursion is produced
by the collision noise and by nothing else. Below the threshold, where the
deterministic model predicts monotone decay, the population still grows by a
factor of order $\kappa_5 / (|\theta| z_{F,0})$ with probability bounded away
from zero, an excursion scale that diverges as the threshold is approached
from below. These excursions are transient (the limit is absorbed at $0$
almost surely for $\theta \le 0$) but their magnitude is not predicted by
the ODE approximation at any distance from criticality that is $O(N^{-\alpha_F})$.
Letting $L \to \infty$ recovers the extinction
probability, to which we now turn in Corollary~\ref{cor:stoch-kessler}.

\begin{corollary}[Stochastic Kessler Syndrome]
    The Feller diffusion has the explicit extinction probability
\[
\pr\big(z_F(\tau) = 0\big) = \exp\!\left( -\frac{\theta\, z_{F,0}}{\kappa_5\,\big(1 - e^{-\theta \tau}\big)} \right),
\]
which at exact criticality ($\theta = 0$) reduces to $\exp\!\big(-z_{F,0}/(\kappa_5 \tau)\big)$. Hence:
\begin{itemize}
    \item \emph{Exactly critical ($\theta = 0$):} the debris population goes extinct almost surely, but the time to extinction is heavy-tailed, $\pr(z_F(\tau) > 0) \sim z_{F,0}/(\kappa_5 \tau)$, and excursions scale like $\sqrt{\tau}$.
    \item \emph{Subcritical ($\theta < 0$):} extinction almost surely, at an exponential rate.
    \item \emph{Supercritical ($\theta > 0$):} survival with positive probability $1 - e^{-\theta z_{F,0}/\kappa_5}$, on which event $\E[z_F(\tau)] = z_{F,0}e^{\theta \tau}$ grows exponentially --- a stochastic Kessler runaway.
\end{itemize}
\label{cor:stoch-kessler}
\end{corollary}
At the threshold, then, whether the debris field runs away is genuinely random, governed by the continuous-state branching process rather than by the ODE drift. 
Indeed, the exponent in the probability has a direct interpretation in terms of the reproduction number $R_0 = \frac{\kappa_3 z_I^\star}{\kappa_5}$. In the near-critical family, a single fragment founds a `lineage' or branching of fragments that survive indefinitely with probability $1-R_0^{-1} \sim R_0 - 1$ for $R_0 \approx 1$. Therefore, the $X_F(0) = N^{\alpha_F} z_{F,0}$ initial fragments produce
\[
    \frac{\theta z_{F,0}}{\kappa_5} = (R_0 - 1) X_F(0)
\]
surviving lineages in expectation. As lineages are independent in a branching process, the number of survivors is a binomial random variable $\text{Bin}(X_F(0), R_0^{-1}(R_0-1))$. Extinction is the event that all $X_F(0)$ fail to produce a surviving lineage, and this happens with probability $(1-R_0^{-1}(R_0-1))^{X_F(0)}$ which is approximately $\exp(-(R_0-1)X_F(0))$. 
Compare this result to the corresponding result for the ODE approximation in Corollary~\ref{cor:kessler}.

Runaway fragment creation thus requires only one lineage among many to escape, and the excess reproduction $R_0 - 1$ trades off against the size of the initial fragment count $X_F(0)$.

\begin{remark}[Logistic branching and immigration at the boundary]\label{rem:logistic}
Each gap of (A2') may be relaxed to equality, at which the corresponding subdominant feedback survives the branching scale. Writing
\[
c := \frac{\kappa_3^2 z_I^{\mathrm{crit}}}{\kappa_4}\lim_{N} N^{\alpha_F - (\rho_1 - \rho_3)} \in [0, \infty),
\qquad
\beta := 2\kappa_2 (z_I^{\mathrm{crit}})^2 \lim_{N} N^{\alpha_F - (\rho_5 - \rho_2)} \in [0, \infty),
\]
the general near-critical limit of Theorem~\ref{thm:critical} is the \emph{logistic continuous-state branching diffusion with immigration}
\[
dz_F = \big(\theta\, z_F - c\, z_F^2 + \beta\big)\,d\tau + \sqrt{2\kappa_5\, z_F}\;dB ,
\]
with $c > 0$ iff $\rho_1 - \rho_3 = \alpha_F$ and $\beta > 0$ iff $\rho_5 - \rho_2 = \alpha_F$ (and either coefficient is $+\infty$, forcing a faster time-scale, if the corresponding gap is violated). The competition term $-c\, z_F^2$ is the self-limiting feedback by which fragments deplete the intacts that produce them --- more fragments give more intact--fragment collisions, hence fewer intacts and a lower fragment birth rate. It \emph{caps} the runaway: even in the supercritical case $\theta > 0$ the fragment population no longer explodes but converges to a positive stationary law (Lambert~\cite{lambert2005} for the logistic branching diffusion; Kawazu and Watanabe~\cite{kawazu1971} for the immigration term). The immigration $\beta > 0$ removes the absorbing boundary at $0$, so the debris persists rather than going extinct. Physically, the intact-depletion feedback supplies an intrinsic ceiling on Kessler-syndrome runaway that the fixed-parameter analysis of Corollary~\ref{cor:kessler} does not see.
It is important to note that the finite Kessler runaway in this remark is a very special situation, and that it is unlikely that the parameters in the real LEO environment will meet the criteria in this remark.
\end{remark}

\section{AI Acknowledgement}
Claude Code (Anthropic, Claude Opus) was used to assist with numerical scripting, figure generation, and manuscript editing.  All mathematical results, modeling assumptions, and scientific conclusions are the authors' own.  No AI tool was used to generate or interpret scientific results.

\bibliographystyle{plain}
\bibliography{references}

\begin{appendix}
\section{Background}
For the sake of readability for a broad audience, we briefly review some of the probability theory used in this work.

\subsection{Poisson Point Processes}

A \emph{point process on $\bR$} is a non-decreasing, countable collection of points of $\bR$.
A point process is \emph{simple} if, almost surely, there is at most one point per location in $\bR$.
A simple point process is \emph{renewal} if the gaps between successive points are independent and identically distributed, and \emph{stationary} if the distribution of the number of points in an interval depends only on the length of the interval.

A \emph{unit-intensity} Poisson process is a renewal process with gap distribution $\Exp(1)$.
Equivalently, a point process is a unit-intensity Poisson process if it is stationary, the distribution of the number of points in an interval of length $L$ is $\Poi(L)$, and the numbers of points in two disjoint intervals are independent.

We will use the more general Poisson process with \emph{intensity measure} $\Lambda$: the numbers of points in disjoint intervals are independent, and for the interval $I$, the number of points is distributed as $\Poi(\Lambda(I)).$
Here, the intensity measure could itself be determined by a stochastic process.

The following property is fundamental and we use it heavily.
If $Y(t)$ is a Poisson process with intensity measure $\Lambda$, and $Z$ is a unit-intensity Poisson process, then for all $t \geq 0$, we have $Y(t) \sim Z(\int_{0}^{t}\Lambda(s)ds)$.
This is called the \emph{time-change representation}.

\subsection{Markov Processes}
A \emph{Markov process} is a stochastic process whose future evolution depends only on its current state: that is, how it got to that state has no impact on its future evolution.
Let $(\Omega, \cF, \pr)$ be a probability space upon which a stochastic process $X$ is defined, and $\{\cF_t\}_{t \geq 0}$ the canonical filtration.
The Markov property states that for $s \geq t$,
\[\pr(X(s) \in A \mid \cF_t) = \pr(X(s) \in A \mid X(t)).\]
Due to the memoryless property of the exponential distribution, Poisson processes are Markov.

In this paper, we study Markov processes with the property that instantaneous transition rates depend on the current state.
The relevant subclass of these processes is \emph{density-dependent Markov processes}.

Concretely, the processes we consider are continuous-time Markov chains on a countable state space --- for us a subset of $\bZ_{\geq 0}^2$, recording the counts of intacts and fragments. From a state $x$, the process holds for an exponentially distributed time and then jumps to $x + \zeta$ at rate $\lambda_\zeta(x)$, for finitely many jump vectors $\zeta$. Its law is encoded by the \emph{infinitesimal generator} $\cA$, which acts on bounded functions $f$ by
\[
\cA f(x) = \sum_\zeta \lambda_\zeta(x)\big[f(x + \zeta) - f(x)\big] ,
\]
and gives the instantaneous mean rate of change of $f$ along the process; equivalently, $f(X(t)) - \int_0^t \cA f(X(s))\,ds$ is a martingale (\emph{Dynkin's formula}).

Our model is specified as a \emph{chemical reaction network}: reactions indexed by $k$, each with a net-change vector $\zeta_k$ and a state-dependent rate, or \emph{propensity}, $\lambda_k(x)$, so that $\cA f(x) = \sum_k \lambda_k(x)\big[f(x + \zeta_k) - f(x)\big]$. Applying the time-change representation for Poisson processes (next subsection) to each reaction, such a process has the pathwise \emph{random time-change representation}
\[
X(t) = X(0) + \sum_k \zeta_k\, Y_k\!\left(\int_0^t \lambda_k(X(s))\,ds\right) ,
\]
with $Y_1, Y_2, \ldots$ independent unit-intensity Poisson processes. Writing $\widetilde Y_k := Y_k - \mathrm{id}$ for the compensated (mean-zero) process, each summand splits into a predictable \emph{drift} and a mean-zero \emph{martingale},
\[
\zeta_k\, Y_k\!\Big(\!\int_0^t \lambda_k\,ds\Big)
= \zeta_k\!\int_0^t \lambda_k(X(s))\,ds
\;+\; \zeta_k\, \widetilde Y_k\!\Big(\!\int_0^t \lambda_k\,ds\Big),
\]
the martingale having predictable quadratic variation $\zeta_k^2 \int_0^t \lambda_k(X(s))\,ds$. This drift--martingale decomposition is the starting point for the fluid and diffusion limits.

Finally, a family of such processes is \emph{density-dependent} when the propensities scale with a large parameter $N$ --- a system size or abundance --- so that, after rescaling the state by powers of $N$, the propensities are of order $N$ times a fixed function of the rescaled state. In the classical setting a single exponent governs every species; the multiscale model of this paper assigns each species its own abundance exponent, which is precisely what produces the separation of time-scales.

\subsection{Fluid and Diffusion Limits}
Let $\{X^N\}_N$ be a family of density-dependent Markov processes indexed by the scaling parameter $N$, and let $Z^N := N^{-\alpha} X^N$ be the rescaled process (componentwise, with abundance exponents $\alpha$), normalized to be of order one as $N \to \infty$. Two complementary limits describe such a family.

\emph{Fluid limit.} In the drift--martingale decomposition of the previous subsection, the martingale carries the prefactor $N^{-\alpha}$ but only the \emph{square root} of the reaction count, so it is of order $N^{-\alpha/2}$ --- smaller than the $O(1)$ drift. When $\alpha > 0$ the noise therefore vanishes as $N \to \infty$, and $Z^N$ concentrates on the solution of a deterministic ordinary differential equation,
\[
\dot z(t) = F(z(t)), \qquad F(z) = \sum_k \zeta_k\, \bar\lambda_k(z) ,
\]
where $\bar\lambda_k$ is the limiting rescaled propensity and $F$ the mean-field vector field. 

\emph{Diffusion limit.} The fluctuations of $Z^N$ about its fluid limit, magnified by $N^{\alpha/2}$, converge to a Gaussian process. Writing $V^N := N^{\alpha/2}(Z^N - z)$, one has $V^N \Rightarrow V$ in the Skorokhod space $D_{\bR^d}[0,\infty)$ of c\`adl\`ag paths --- with $\Rightarrow$ denoting weak convergence --- where $V$ solves the linear stochastic differential equation
\[
dV(t) = DF(z(t))\, V(t)\,dt + \sigma(z(t))\,dB(t) ,
\]
with $DF$ the Jacobian of $F$, $B$ a standard Brownian motion, and diffusion matrix $\sigma\sigma^{\top}(z) = \sum_k \zeta_k \zeta_k^{\top}\, \bar\lambda_k(z)$ read off from the martingale quadratic variations. 
The limit $V$ is an Ornstein--Uhlenbeck-type process, and this Gaussian approximation is the \emph{diffusion limit}.

These two regimes are not exhaustive. When the fluid vector field vanishes to leading order --- as it does at a critical point --- the $N^{\alpha/2}$ centering degenerates, and the fluctuations, gathered over a correspondingly longer time-scale, converge instead to a nonlinear diffusion. For the branching structure of the fragment population this limit is a Feller diffusion rather than a Gaussian one; see Section~\ref{sec:critical}. Determining which of these limits governs each species, and on which time-scale, is the main concern of this paper.

\end{appendix}

\end{document}